\documentclass [12pt,a4paper,reqno]{amsart}
\input amssymb.sty

\newtheorem{thm}{Theorem}[section]

\newtheorem{cor}[thm]{Corollary}
\newtheorem{note1}[thm]{Note}
\newtheorem{defn}[thm]{Definition}

\newtheorem{exmpl}[thm]{Example}

\newtheorem{lem}[thm]{Lemma}
\newtheorem{prop}[thm]{Proposition}

\newtheorem{rem}[thm]{Remark}

\newcommand\Dref[1]{{Definition~\ref{#1}}}
\newcommand\Rref[1]{{Remark~\ref{#1}}}
\newcommand\Tref[1]{{Theorem~\ref{#1}}}
\newcommand\Cref[1]{{Corollary~\ref{#1}}}
\newcommand\Pref[1]{{Proposition~\ref{#1}}}
\newcommand\Eref[1]{{Example~\ref{#1}}}

\newcommand\Ssref[1]{{Subsection~\ref{#1}}}

\newcommand\M[1][d]{{\operatorname{M}_{#1}}}
\def\({\left(}
\def\){\right)}
\newcommand{\smat}[4]{{\(\!\!\begin{array}{cc}{#1}\!&\!{#2}
    \\[-0.1cm]{#3}\!&\!{#4}\end{array}\!\!\)}}
\newcommand{\Mat}[9]{{\( \begin{array}{ccc} {#1} & {#2} & {#3} \\ {#4} & {#5} & {#6} \\ {#7} & {#8} & {#9} \end{array} \)}}

\def\indecomposible{indecomposable}
\def\id{\operatorname{id}}

\def\Hom{\operatorname{Hom}}
\def\End{\operatorname{End}}

\def\a{{\alpha}}
\def\lam{{\lambda}}
\def\la{{\lambda}}
\def\vareps{{\varepsilon}}
\def\sub{{\,\subseteq\,}}
\newcommand{\set}[1]{{\left\{#1\right\}}}
\newcommand{\ideal}[1]{{\left<{#1}\right>}}

\newcommand\eq[1]{{(\ref{#1})}}
\newcommand{\card}[1]{{\left|{#1}\right|}}

\def\cha{\operatorname{char}}
\def\pol{\operatorname{poly}}
\def\cent{\operatorname{Cent}}

\DeclareMathOperator{\Rad}{Rad}
\newcommand\cl[1]{{#1}^{\operatorname{cl}}}
\def\II{{I\!\!\,I}}
\def\III{{I\!\!\,I\!\!\,I}}
\def\IV{{I\!\!\,V}}
\def\co{{\,{:}\,}}

\def\ra{{\rightarrow}}
\def\N {{\mathbb {N}}}
\def\F {{\mathbb {F}}}
\def\Z {{\mathbb {Z}}}

\newcommand\dimcol[2]{{[{#1}\!:\!{#2}]}}
\newcommand\suchthat{{\,:\ \,}}
\newcommand\subjectto{{\,|\ }}

\def\isom{{\;\cong\;}}

\newcommand\compfull[3][\bullet]{{{#1}_{#2(#3)}}}
\newcommand\comp[3][\bullet]{{{#1}_{{\if1#2{}\else{#2}\fi}{\if#3K{}\else{(#3)}\fi}}}}

\newif\ifXY
\XYtrue
\ifXY \usepackage{xy}\fi
\ifXY \xyoption{matrix}\xyoption{arrow}\xyoption{curve} \fi

\def\Zcd{{Zariski closed}}
\def\Zcr{{Zariski closure}}

\usepackage[usenames]{color}

\begin{document}
\title[Full quivers of representations of algebras]
{Full quivers of representations of algebras}

\date{\today}

\author{Alexei Belov-Kanel}
\author{Louis H. Rowen}
\author{Uzi Vishne}

\address{Department of Mathematics, Bar-Ilan University, Ramat-Gan
52900, Israel}
\email{\{belova,rowen,vishne\}@macs.biu.ac.il}

\thanks{This research was supported by the
Israel Science Foundation, (grant No. 1178/06).}

\begin{abstract}
We introduce the notion of the full quiver of a representation of
an algebra, which is a cover of the (classical) quiver, but which
captures properties of the representation itself. Gluing of
vertices and of arrows enables one to study subtle combinatorial
aspects of algebras which are lost in the classical quiver. Full
quivers of representations apply especially well to \Zcd\
algebras, which have properties very like those of finite
dimensional algebras over fields. By choosing the representation
appropriately, one can restrict the gluing to two main types: {\it
Frobenius} (along the diagonal) and, more generally {\it
proportional} Frobenius gluing (above the diagonal), and our main
result is that any representable algebra has a faithful
representation described completely by such a full quiver. Further
reductions are considered, which bear on the polynomial
identities.
\end{abstract}

\maketitle

\tableofcontents

\newcommand\LL[2]{{\stackrel{\mbox{#1}}{\mbox{#2}}}}
\newcommand\LLL[3]{\stackrel{\stackrel{\mbox{#1}}{\mbox{#2}}}{\mbox{#3}}}
\newcommand\LLLL[4]
{\stackrel {\stackrel{\mbox{#1}}{\mbox{#2}}}
{\stackrel{\mbox{#3}}{{\mbox{#4}}}} }

\newcommand\AR[1]{{\begin{matrix}#1\end{matrix}}}

\section{Introduction}

This paper is part of an ongoing project, elaborating \cite{B2},
\cite{B5}, to lay a firm foundation for Belov's positive solution
for Specht's problem for affine algebras with polynomial identity
(PI-algebras) in characteristic $p>0$; cf.~\cite{B2}. Briefly, two
algebras are called PI-{\it equivalent} if they satisfy the same
PIs, and the major question in PI-theory is to classify
PI-equivalence classes of algebras, thereby classifying {\it varieties}.
The first question in this direction is {\it Specht's problem},
which asks whether every variety is determined by a finite set of
PIs. Kemer \cite{Kem88}, \cite{Kem90} solved Specht's problem in
characteristic 0 for arbitrary algebras, although Specht's problem
has counterexamples in characteristic $p$, cf.~\cite{B1},
\cite{G}. One of Kemer's main results was that every affine
PI-algebra over an infinite field is PI-equivalent  to a finite
dimensional (f.d.)~algebra, and thus combinatoric properties of
finite dimensional algebras could be used to finish the solution
of Kemer's problem in this case.

In general, we recall \cite[pp.~28ff.]{BR} that an algebra $A$
over an integral domain $C$ is {\bf representable} if it can be
embedded as a $C$-subalgebra of $\M[n](K)$ for a suitable field
$K\supset C$ (which can be much larger than $C$). Besides their
intrinsic importance in many aspects of algebra, representable
algebras obviously satisfy the PIs of $n\times n$ matrices, and
are the major examples of PI-rings. One  of Kemer's theorems
states that, over an arbitrary infinite field, every relatively
free affine PI-algebra  is representable, and thus every variety
of PI-algebras contains a representable algebra.

Since the situation for affine PI-algebras over arbitrary fields
is more complicated, one is led to search representation theory
for useful tools.  Ever since the pioneering works of Gabriel
\cite{Ga} and Bern\v{s}te\u\i{}n-Gelfand-Ponomarev \cite{BGM}, one
of the most important such tools has been the quiver of a finite
dimensional algebra.

One major difficulty which arises immediately is that quivers are
most easily applied to algebras over algebraically closed fields.
Serious complications arise for algebras over finite fields. For
example, let $\F _q$ denote the field of $q$ elements, where $q$
is a power of the prime number $p$. Whereas over an infinite
field, any affine PI-algebra is PI-equivalent to a f.d.~algebra,
this assertion fails for PI-algebras over the finite field $F= \F
_q$,   as exemplified by the algebra $\(\begin{array}{cccc}
F & F[\la]  \\
0 & F[\la]
\end{array}\)$.

In this way, we are led to consider a more detailed combinatoric
object, the {\bf full quiver} of a representation of an algebra,
which turns out to be a cover of the usual quiver. Although much
of the theory can be formulated for algebras over arbitrary
commutative affine rings, we limit our attention mostly to
algebras over an (arbitrary) field $R$.

We focus on {\bf Zariski closed algebras}, described in
\cite{BRV1} and reviewed below. All f.d.~algebras are Zariski
closed, and Zariski closed algebras satisfy the basic structure
theorems that Wedderburn proved for f.d.~algebras. In particular,
the radical $J$ of a \Zcd\ algebra $A$ is nilpotent, and one can
write $A = S \oplus J$ where $S$ is a semisimple subalgebra of
$A$. The vertices of the full quiver of a representation of $A$
correspond to the diagonal blocks arising from simple components
of the semisimple part $S$, whereas the arrows correspond to
linear functions between diagonal blocks, and
 come from the radical $J$.

As noted in \cite[Lemma 3.18]{BRV1}, any algebra is PI-equivalent
to its \Zcr. Thus, the PI classification problem reduces to
classifying \Zcd\ algebras of relatively free algebras in terms of
PIs. In this paper we study full quivers of representations, with
special attention to finding the quiver having the best form; our
interest lies in how the full quiver describes the interaction
between the radical of an algebra and the semisimple part. The
applications to PI-theory are given in \cite{BRV3}.

Our study of full quivers is based on the structure of \Zcd\
algebras. In \cite{BRV1} we showed that any \Zcd\ $F$-subalgebra
$A$ of $\M[n](K)$ (where $K$ is an algebraically closed field
containing $F$) has a {\it Wedderburn decomposition} in which it
is a direct sum of the radical and the semisimple part;
furthermore, $A$ can be represented in {\it Wedderburn block
form}, in which diagonal blocks  comprise the semisimple part, and
the radical embeds above the diagonal (see \Dref{Wbf}), with
certain identifications of the blocks which we call {\it gluing}.
 Gluing occurs separately for the diagonal (semisimple) and the off-diagonal (radical) components.

Diagonal gluing is very easily described, since any identification
of diagonal blocks can be viewed as an isomorphism of matrix
 algebras, which in turn is described in terms of an
isomorphism of their  fields of scalars. But the automorphisms of
finite fields are given by powers of the Frobenius map $a\mapsto
a^p$, where $p$ is the characteristic, so all diagonal gluing can
be described in terms of {\it Frobenius gluing} of the diagonal
blocks, as illustrated in the following example: When $\cha (F) =
p$, we have the Frobenius automorphism $\a \mapsto \a ^{p^u}$ of
$F$, and one can construct the algebra
$\bigg\{\smat{\alpha^{p^u}}{*}{0}{\alpha}: \alpha \in F\bigg\}$
(which satisfies the PI $x[y,z] - [y,z]x^{p^u}$). We also call
$p^u$ the {\bf Frobenius twist} in the gluing. When $u=0$, i.e.,
when this automorphism is the identity, we say there is no
Frobenius twist, and call the gluing {\it identical gluing}.
\cite[Theorem~5.14]{BRV1} (given below as Theorem~\ref{Frobglue})
describes all gluing along the diagonal blocks.

Thus,  our theory reduces to the description of the blocks above
the diagonal, in terms of the full quiver.  Our overall objective
in this paper is to use the full quiver to determine various
properties of the underlying algebras, especially the interaction
between the semisimple part and the radical. This task is fairly
easy when the quiver is straightforward enough. The simplest kind
of quiver, consisting of a single arrow and called an {\it
elementary quiver}, already displays such interactions.

There also can be gluing above the diagonal, i.e., among arrows,
which gives rise to relations among the radical elements. We call
this \textbf{off-diagonal gluing}. Off-diagonal gluing of two
arrows entails gluing of their initial vertices and terminal
vertices. For example, one  can have identical off-diagonal gluing
of two arrows whose initial vertices have Frobenius gluing on the
diagonal; this indicates an identification of components via a
Frobenius automorphism of the underlying field.

A certain kind of gluing, called {\it{compression}}, enables us to
shrink the size of the representation, at the expense of replacing
the base field $F$ by an $F$-algebra with nilpotent elements,
which we call {\it{infinitesimals}}.

Another kind of (off-diagonal) gluing  is called {\it Frobenius
proportional}, by which we mean that gluing between two radical
components is by means of some scalar multiple (perhaps involving
a Frobenius automorphism of the initial or terminal vertices of
the arrows). If these Frobenius automorphisms are trivial, we
define the gluing to be {\it purely proportional}. But one could
also have other sorts of gluing which are much more complicated
and very difficult to describe. Different representations of the
same algebra will produce different full quivers, which on the one
hand leads to ambiguity, but also permits us to change the
representation to obtain a quiver of better form.

The main goal of this paper is to obtain the full quiver in the
best form.  Since the Frobenius twist only occurs for $F$ finite,
our theory becomes much easier over infinite fields. To obtain the
sharpest possible results,  we must limit our attention to the
full quiver of relatively free algebras. Towards this end, we have
the following theorems:

\begin{thm}[cf.~\Tref{propglu}]\label{main}
The \Zcr\ of any representable affine PI-algebra over an
infinite field has a representation and full quiver all
of whose polynomial relations  are consequences of
proportional gluing.
\end{thm}

\bigskip
\begin{thm}[cf.~\Tref{propglurel}]
Any relatively free affine PI-algebra over an infinite field has a
representation whose quiver has no double edges (between two
adjacent vertices) in the full quiver. More generally, any two
branches between two vertices are proportionally permuted  (see
\Dref{permglu}).
\end{thm}

\bigskip
\begin{thm}[cf.~\Tref{propglu11}]\label{main11}
Any relatively free affine PI-algebra has a representation for
whose full quiver all gluing is Frobenius proportional.
\end{thm}

We stress that all the polynomial relations of an algebra are
consequences of gluing relations of its full quiver. (For example, in
Theorem~\ref{propglu}  they must be consequences of proportional
gluing). Due to interactions of quiver branches for finite fields, in
order to obtain the sharpest reformulation of
Theorem~ref{propglurel1}, we need to introduce gradings in the
forthcoming paper \cite{BRV3}.

In the appendix, we introduce the notion of pseudo-quivers, to
prove the following result:

\bigskip
\begin{thm}[cf.~\Tref{nilpcan}]
Any \Zcd\ algebra has a representation for whose full quiver the maximal length of
the branches equals the index of nilpotence of the radical minus~1.
\end{thm}

A maximal path in the full quiver is called a {\it
branch}. Gluing between branches is called {\it total}
if for each arrow in one branch there is an arrow of the other
branch, glued  proportionally to it. One candidate for the best
form of a full quiver is called {\it \indecomposible},
cf.~Definition~\ref{candef1}.

In a subsequent paper \cite{BRV3}, we discuss canonization
theorems for full quivers (as well as pseudo-quivers) and apply
this theory towards the classification of varieties of
PI-algebras, stressing the role of pseudo-quivers. This paper
contains the part of the theory concerning the representations of
algebras; in \cite{BRV3} we introduce techniques of evaluating
polynomials by means of the full quivers, thereby enabling us to
further improve full quivers of relatively free algebras.

This theory thereby enables us to obtain clearer proofs of some
results about varieties of PI-rings and also to motivate and unify
several interesting examples in PI-theory.

One specific application in \cite{BRV3}: We define a {\bf
parametric} set of identities of an algebra~$R$ to be a set of
identities of the form
$$\left\{ \sum \xi_j f_j  :  \xi_j \in F \right\},$$ where the $f_j$   are homogeneous polynomials that are not
identities of $R$.

Branches in the quiver, together with gluing, often define a
parametric set of PIs, although these identities are not always
definable over the base field $F$ but sometimes require passing to
a transcendental field extension of $F$. Other identities arise
from gluing among different paths of the canonical quiver.

In another direction, one can use quivers to study ring-theoretic
properties of relatively free PI-algebras; for example, in
\cite{BRV3} we prove the following result:

\bigskip

 \textbf{Theorem.} {\it A relatively free   affine algebra $A$ is weakly
Noetherian iff it has a faithful representation for which each
connected component of the full quiver consists of a single arrow,
with each vertex having matrix degree at most 1.}

\section{Review of Zariski closed algebras}

In this section we review Zariski closed algebras, studied in
\cite{BRV1}, which form the foundation for the study of
representable PI-rings. Suppose $F \sub K$ is an algebraically
closed field.  Let $\rho \co A \to \M[n](K)$ be a representation
of a ring $A$. The {\bf \Zcr} $\cl{\rho(A)}$ is the closure of
$\rho(A)$ with respect to the Zariski topology of $\M[n](K)$.

This definition clearly depends on the choice of the
representation $\rho$. Nevertheless, abusing language slightly, we
assume that $\rho$ is a faithful representation, taken as given,
and thus we view the ring $A$ as contained in $\M[n](K)$; we call
$\cl {\rho (A)}$ the {\bf \Zcr} of~$A$, denoted as $\cl{A}$. For
convenience, we assume that our representable ring $A$ is an
$F$-algebra.

\begin{lem}\label{nilrad} The radical $J$ of the Zariski closure of
any affine PI-algebra $A_0$ is nilpotent.\end{lem}
\begin{proof} Let $J_0$ be the radical of $A_0$. By
the Braun-Kemer-Razmyslov theorem, $J_0^\ell = 0$ for some $\ell$.
Hence $x_1 \cdots x_\ell$ is an identity of $J_0$, and thus of its
Zariski closure, which is~$J$ by \cite[Proposition 3.21]{BRV1}. In
other words, $J^\ell = 0$.
\end{proof}

Since any $K$-subalgebra of $\M[n](K)$ is
\Zcd, we can take the \Zcr\ of $A$ in any $K$-subalgebra $B
\subseteq \M[n](K)$ that contains $\rho(A)$. In particular, we can
take $B$ to be the {\bf linear closure} $ K A,$ the $K$-subspace
of $\M[n](K)$ spanned by $A$.

When $F$ is infinite, the \Zcr\ is just the linear closure. Thus,
our special interest in the \Zcr\ is for algebras over finite
fields. As shown in \cite{BRV1}, Zariski closed algebras satisfy
many of the structural properties of finite dimensional algebras
over algebraically closed fields.

\begin{rem}\label{BA}
By \cite[Proposition 3.20]{BRV1}, we may assume that every ideal
of $B= KA$ intersects $A$ nontrivially. These assumptions are
implicit in the rest of the text. This ties the structure of $A$
to the structure of $B$.
\end{rem}

Viewing $A \subseteq B$ explicitly as in Remark~\ref{BA}, we fix a
base $v_1,\dots,v_m$ of $B$ over~$K$, where each $v_i \in A$.
Although $A$ need not be a $K$-algebra, every element of $A$ is
written as $\alpha_1 v_1+\dots+\alpha_m v_m$ for suitable
$\alpha_i \in K$. In order to study the coefficients, we study
their set $\pol(A)$ of {\bf polynomial relations}, which are
defined as those polynomials $f \in K[\lam_1,\dots,\lam_m]$ such
that $f(\alpha_1,\dots,\alpha_m) = 0$ for every $\alpha_1 v_1 +
\dots + \alpha_m v_m \in A$. In particular, polynomial relations
must have constant term $0$. One advantage of this point of view
is that an algebra is \Zcd\ if and only if its polynomial
relations serve as its defining relations. On the other hand,
linear transformations are continuous in the Zariski topology,
implying that the \Zcr\ is independent of the choice of base of
$B$.

One well-known connection to PI-theory is that, when we designate
a given base $B$, any PI of the algebra $A$ can be viewed as a set
of polynomial relations in the coefficients of the elements of
$A$, written in terms of the base $B$; cf.~\cite[Proposition
3.17]{BRV1}. Hence the \Zcr\ $\cl{A}$ of $A$ is an algebra, which
is PI-equivalent to $A$, and we usually take our algebra $A$ to be
\Zcd.

\begin{exmpl}\label{fields} {\ }

(i) The \Zcd\ $F$-subalgebras of an algebraically closed field $K$
are precisely the {\emph {finite}} intermediate subfields $F\sub
F_1$ of $K$, as well as $K$ itself. Indeed, any infinite field
cannot have nontrivial polynomial relations, and a finite field
$F_1$ of $q$ elements satisfies the polynomial identity $\la _1
^q-\la _1$. But $\cl{F_1}$ must also satisfy the identity $\la _1
^q-\la _1$ and thus have at most $q$ elements; hence, $\cl{F_1} =
F_1.$

(ii) Other examples of commutative \Zcd\  subalgebras of
$\M[n](K)$ include
$$\set{\(\begin{array}{cccc}
\alpha & \beta & \gamma\\
0 & \alpha & \beta \\
0 & 0 & \alpha
\end{array}\) \suchthat \alpha , \beta, \gamma \in K}\cong K[\lam]/\ideal{\lam^3}.$$
\end{exmpl}

The starting point in the representation theory of finite
dimensional algebras may well be Wedderburn's Principal Theorem.

\begin{thm}[Wedderburn's Principal Theorem]\label{WedPrinc} If $A$ is a f.d.~algebra over an algebraically closed field $K$, then $A$ has a Wedderburn
decomposition $A = S\oplus J,$ where $J = \Rad(A)$ and $S \isom
A/J$ is a subalgebra of $A$ that is a direct sum of matrix
algebras over $K$.\end{thm}

The following theorem, parallel to Wedderburn's principal theorem,
gives us the basic structure of \Zcd\ algebras over arbitrary
fields.

\begin{thm}[First Representation Theorem, {\cite[Theorem 3.33]{BRV1}}]\label{zarcl1}
Let $A$ be a representable $F$-algebra. If $A = \cl{A}$,
then $A$ has a Wedderburn decomposition $A = S\oplus J$, where $J
= \Rad(A)$ and $S \cong A/J$ is a subalgebra of $A$ that is
isomorphic to a direct sum of matrix algebras over fields (which
are closed $F$-subfields of $K$).
\end{thm}

Thus, the basic question in studying these algebras is determining
how the radical~$J$ interacts with the semisimple part $S \cong
A/J$.

\subsection{Polynomial relations of \Zcd\ algebras}\label{ss:Zc}

 We
review some of the results about polynomial relations of \Zcd\
algebras proved in \cite{BRV1}.

\begin{defn} A polynomial of the form $\sum _{i=1}^m
\sum _{j\ge 1} c_{ij} \lam_i^{q_{ij}}$ is called a {\bf
$q$-polynomial}, if each $q_{ij}$ is a $q$-power with $q=
\card{F}$, where we take $q = 1$ if $F$ is infinite.
\end{defn}

In \cite{KombMiyanMasayoshi}, \cite{Miyanishi}, \cite{Tits} it is
shown that every additively closed subvariety of affine spaces is
given by a system of equations each of which is a $q$-polynomial.
In other words, we have the following assertion, which also can be
seen by combining Proposition 4.7 and Theorem 4.10 of \cite{BRV1}:

\begin{cor}\label{addit} $\pol(A)$ is a finitely
generated $F[\phi]$-module (where $\phi$ is the Frobenius map
$\lam_i \mapsto \lam_i^q$). It has a finite set of defining
relations, and every polynomial relation is a consequence of
these.
\end{cor}

\begin{defn}\label{2.8}
A polynomial relation is of {\bf $F$-Frobenius type}
if it has one of the following three forms, where $q$ is as in the
previous definition:

(i) $\lambda _i = 0,$

(ii) $\lambda _i= \lambda _i^s ,$ where $s$ is a $q$-power, or

(iii) $\lambda _i = \lambda _j^s$, $j \ne i$, where $s$ is a
$q$-power.
\end{defn}

We quote \cite[Theorem 4.15]{BRV1}:

\begin{thm}\label{comZar}
Suppose $A$ is a commutative, semiprime
\Zcd\ $F$-subalgebra of a finite dimensional commutative
$K$-algebra $B$. Then $\pol(A)$ is generated by finitely many
polynomial relations of $F$-Frobenius type.
\end{thm}

This theorem is applicable to the center of $A$, since the center
of a \Zcd\ algebra is \Zcd, by \cite[Lemma 3.28]{BRV1}. The next
step is to use this description to build a suitable representation
of $A$. For this, we also recall (and slightly generalize) a definition from \cite{BRV1}.

\begin{defn}\label{Wbf}
Let $F \sub K$ be commutative Noetherian rings. Suppose $A$ is a
$K$-algebra  with an ideal $J_0 \sub \operatorname{Jac}(A)$ (the
radical of $A$), such that $$A/J_0 = A_1 \times \dots \times A_k
\cong \M[n_1](F_1) \times \dots \times \M[n_t](F_k),$$ for
subrings $F \subseteq F_u \subseteq K$.

We say that a representation $\rho\! : A\to \M[n](K)$ is in {\bf
Wedderburn block form} if, for suitable $s_1,\dots,s_k \geq 1$,
the diagonal of $\M[n](K)$ is contained in $s_1+\dots+s_k$ disjoint
diagonal blocks $A_u^{(1)}, \dots, A_u^{(s_u)}$ ($u = 1,\dots,
k$), satisfying the following properties:
\begin{itemize} \item
For each $u$, $A_u^{(i)}$ has size $n_u \times n_u$ for each~$i$,
and is isomorphic to $A_u$.
\item $n = \sum_u s_un_u;$
\item The given representation $\rho$ restricts to an embedding $\rho_u\! :
A_u \to A_u^{(1)}\times \dots \times A_u^{(s_u)}$;
\item The ensuing composite of $\rho_u$ with the projection onto the $i$
component yields an isomorphism $\rho_u^{(i)}\! : A_u \to A_u^{(i)}$ for each $i$;
\item $\rho $ embeds $J$ into the sum of the strictly upper triangular blocks (above the diagonal blocks).
\end{itemize}
For each $u$, the blocks $A_u^{(1)},\dots,A_u^{(s_u)}$ are said to
be {\bf diagonally glued}.
\end{defn}
We usually assume that $F$ is a field  whose   algebraic closure
is $K$, and in this case $J_0=\operatorname{Jac}(A)$.

Thus, each of $A_u^{(1)}, \dots, A_u^{(s_u)}$ has center
isomorphic to $F_u$. Using idempotents, when $A$ is Zariski closed
(with $F\subseteq K$ fields) we can apply Example~\ref{fields}  to
conclude that the $F_u$ are either finite or equal to $K$. If
$F_u$ is finite, we say that the corresponding blocks have {\bf
finite type}.

For example, suppose $F = \F_q$. Then $$A =
\left\{\Mat{\alpha}{0}{\gamma}{0}{\beta}{\delta}{0}{0}{\beta^q}:\quad
\alpha,\gamma,\delta\in K, \ \beta \in \F_{q^2} \right \}$$ is a
\Zcd\ $F$-subalgebra of $\M[3](K)$, defined by the polynomial
relations $\lam_{12} = \lam_{21} = \lam_{31} = \lam_{32} = 0$ and
$\lam_{33} - \lam_{22}^q = \lam_{33} - \lam_{33}^{q^2} = 0$. (The
last two polynomial relations formally imply the polynomial
relation $\lam_{22}-\lam_{22}^{q^2}= 0$.) In this case $A/J \isom
K\oplus \F_{q^2}$, where the second component, which is finite,
embeds via $\beta \mapsto (\beta,\beta^q)$, thus gluing
the~$e_{22}$ and $e_{33}$ entries.

Here is an easy but important special case.

\begin{rem}\label{semis1} Assume $A \cong A_1 \times \cdots \times A_k$ is semisimple;
i.e., $J = 0$. Then the only Wedderburn blocks are along the
diagonal. Any representation of $A$ restricts to a representation
of the center, whose polynomial relations have already been
described in~Theorem~\ref{comZar}. These polynomial relations
extend to the respective entries of the matrix algebra components,
so we conclude that all gluing is diagonal, and is either
identical gluing or Frobenius gluing.
\end{rem}

For arbitrary \Zcd\ algebras, we start with Remark \ref{semis1}, since the semisimple part is a subalgebra, but also may have
$J \ne 0,$ in which case we also need to contend  with blocks
above the diagonal. The following result from \cite[Theorem~5.14]{BRV1} describes the relations on the diagonal.

\begin{thm}\label{Frobglue}
Suppose $A \subseteq \M[n](K)$ is a \Zcd\ algebra, with
$$A/J = A_1 \times \cdots \times A_k,$$ a direct product of $\, k$
simple components. Then we can choose the matrix units of~
$\M[n](K)$ in such a way that $A$ has Wedderburn block form, and
all identifications among the diagonal blocks are Frobenius
gluing.
\end{thm}

Recalling the linear closure $B = KA$, we write the diagonal
components of $B$ as $B_1, B_2, \dots, B_\ell$ (taken in
$M_n(K)$). Letting $e_r$ denote the unit element for the component
$B_r,$ for $1 \le r \le \ell,$ we have $B_r = e_r B e_r;$ we also
write $B_{r,r'}$ for $e_r B e_{r'}$. (Thus, $B_{r,r} = B_{r}$.)
Then
$$B = \bigoplus \limits _{r,r'=1}^k B_{r,r'}.$$ We carry this
notation throughout.

 The relations above the diagonal are subtler. It is particularly useful to describe the
Wedderburn block form in terms of polynomial relations. Taking the
matrix units as a base for $ \M[n](K)$, the relations module
$\pol(A)$ of $A$ in Wedderburn block form can be decomposed as a
direct sum, in accordance with the decomposition $A = S \oplus J$.
Thus $A$ is defined by three classes of relations on the
components $A_{r,r'} = A \cap B_{r,r'}$: (Here we write $A_r$ for
the block $A_{r,r};$ this would correspond to some $A_u^{(i)}$ in
the notation of Definition~\ref{Wbf}.)

\begin{enumerate}
\item\label{C1} $\lam_{ij} = 0$ for $i>j$ (which says that all entries below the
diagonal blocks are 0).
\item\label{C2} Relations of the form $\lam_{ij} =
\lam_{i'j'}^{q^t}$, where $(i,j)$ and $(i',j')$ are in the same
relative  position in glued blocks; we call this {\bf Frobenius
gluing} of {\bf exponent} $ t$. When diagonal blocks $A_r$ and
$A_{r'}$ have Frobenius gluing of exponent $t$, we write
$t=\exp(A_{r,r'})$, and call it the {\bf relative Frobenius
exponent}. We permit $t=0;$ in fact we must have $t=0$ when the
base field $F$ is infinite. For $F$ finite, $\exp(A_{r,r'})$ is
only well defined modulo the dimension of $F_r = F_{r'}$ over $F$
\item\label{C3} Other relations  link components above the
diagonal blocks by means of $q$-polynomials.
\end{enumerate}

\begin{note1}
Occasionally, we want to deal with algebras without 1, so let us
consider this case briefly. Any finite dimensional algebra $B$
without $1$ has a maximal set of orthogonal primitive idempotents
$\{e_1, \dots, e_k\}$, whose sum we denote as $e$. We formally
define the operator $e_0$ by $e_0 b = b - eb$ and $be_0 = b - be$.
Clearly the operator $e_0$ is idempotent, and $e_0 b = 0$ iff $b
\in eR$.
$B_{0,0} = e_0Be_0$ must be nilpotent, since otherwise it would
contain a nonzero idempotent of $B$, contrary to hypothesis.
Strictly speaking, $B_0$ may be composed of several nilpotent
blocks, as in the example $B = \set{\(\begin{array}{cccc}
0 & *& 0 & 0\\
0 & 0 & 0 & 0\\
0 & 0 &0 & *\\
0 & 0 & 0 & 0
\end{array}\)}.$

Using this Peirce decomposition, one can obtain a corresponding
Wedderburn block decomposition, but where one of the diagonal
blocks is $B_{0,0}$. The zero block can have arbitrary dimension.

For example, when $B$ is nilpotent,  the idempotent $e = 0$ and
$e_0$ is the identity operator and $e_0 Be_0 = B$. For $B =
\set{\(\begin{array}{cccc}
0 & 0 & *\\
0 & 0 & *\\
0 & 0 & *
\end{array}\)},$ the block $B_0$ is a $2 \times 2$ zero block.
\end{note1}

We would like to focus on one particular kind of gluing of
non-diagonal blocks.
\begin{defn}\label{propFb} {\bf Proportional Frobenius gluing} is the situation
in which, for suitable powers $q,q'$ of $p$, the $q$ power of the
entries in $B_{r',s'}$ each are $\nu$ times the $q'$ power of the
respective entries in $B_{r,s}$. When $q,q'=1,$ this is merely
called {\bf proportional gluing}.  When moreover $\nu =1$, we say
that the component has {\bf identical gluing}.\end{defn}

Note that proportional Frobenius gluing is transitive. Other
relations of class \eq{C3} often are far more complicated to
describe than those of class \eq{C1} and~\eq{C2}. Fortunately, we
may bypass them, as we shall see.

\begin{defn}
We define the {\bf glue equivalence relation} on the diagonal
blocks, by saying $B_{r}\sim B_{r'}$ when these two diagonal
blocks are glued. Thus, for $k$ as in~Definition~\ref{Wbf}, there
are $k$ equivalence classes of glued diagonal blocks, which we
call {\bf glued components}.
\end{defn}

\begin{rem}\label{offdiag} The relative Frobenius exponents are used to define equivalence
relations on vectors of glued indices, as follows: Let $T_u$
denote the set of those indices $r \in \{ 1, \dots, n\}$ such that
the diagonal block $B_r$ belongs to the $u$ glued component.
 For every $1\leq r,s \leq k$, we let
$T_{u,v} = \set{(r,s) \in T_u \times T_v \suchthat r\leq s}$, and
define an equivalence relation on $T_{u,v}$ by setting $(r,s) \sim
(r',s')$ iff $\exp(B_{r,r'}) \equiv \exp(B_{s,s'})$ modulo
$\gcd(t_r,t_s)$, where $t_r$ is the dimension of $\cent (B_{r})$.

(Formally we have a partial matrix $(\ell_{rs})$, defined when
$B_r$ and $B_s$ are glued, such that $\ell_{rr} = \ell_{ss}$ is
the dimension of $F_r$ over $F$, and $\ell_{rt} =
\ell_{rs}+\ell_{st} \pmod{\ell_{rr}\Z}$ if $B_r$, $B_s$ and $B_t$
are glued.)

\end{rem}

\subsection[The role of Peirce decompositions]{The role of Peirce decompositions}

\begin{rem}\label{Peirce0} One of the primary tools in the analysis of a \Zcd\ algebra $A$ is
the Peirce decomposition, which we use here at three levels.

\begin{itemize}
\item We start with the Peirce decomposition of $A$. The gluing
equivalence does not recapture the Wedderburn principal
decomposition of Theorem \ref{zarcl1}; in Example \ref{fields},
all the $e_r$ are glued, and their sum is the unit element $ 1.$
In general, a local algebra has no idempotents, so we cannot
extract its semisimple part from its Peirce decomposition!

\item The next refinement is the Peirce decomposition of the
linear closure $B = KA.$ We call this the \textbf{sub-Peirce
decomposition} of $A$. Note that non-identity Frobenius gluing in
$A$ becomes ``unglued'' in $B$, when viewed as an algebra over the
infinite field $K$.

\item Finally,  each idempotent $e$ of $B$ corresponds to a glued
component, and we can write $e$ as a sum of unit elements of the
glued blocks of this component. This is the finest Peirce
decomposition with which we work.
\end{itemize}
\end{rem}

(Other decompositions which come up, especially in the proof of
Theorem~\ref{propglu11}, arise from Frobenius gluing, and from
algebraic group decompositions.)

\begin{rem}\label{decomp} The Peirce decomposition of an algebra $A$ can be viewed
structurally, in terms of decomposing $A = \sum Ae_i$ as a direct
sum of projective modules; then $$A \approx \End _A A = \End _A
\left(\sum Ae_i\right) \cong \bigoplus \Hom _A (Ae_i, Ae_j) \cong
\bigoplus e_i A e_j.$$ Since $A \subset B \subset M_n(K)$ in
Remark~\ref{Peirce0}, this observation can be applied to $A$ as a
subalgebra of any of these three algebras, and  $A \cap \Hom
(M_n(K)e_i, M_n(K)e_j)$ is contained in the direct sum of suitable
Wedderburn blocks. Note that each Wedderburn block is an
$M_{n_i}(F_i)$-$M_{n_j}(F_j)$-bimodule in general.
\end{rem}

\begin{exmpl}\label{zarcl2}  We can recover the decomposition $A = S \oplus J$
of Theorem~\ref{zarcl1} from the intersection of $A$ with the
Wedderburn block components of $B = KA$. In other words, matching
components yields $$S = (\bigoplus e_r B e_r) \cap A = (\bigoplus
e_r A e_r) \cap A;$$
$$J = (\bigoplus _{r < s} e_r B e_s) \cap A =
(\bigoplus _{r < s} e_r A e_s) \cap A.$$
\end{exmpl}

\section{Full quivers of representations}\label{quivers2}

We now turn to our main objective, which is to describe
representations of algebras in terms of quivers (i.e., directed
graphs). First we review the classical theory of quivers; an
accessible reference is \cite{ASS}. Then we introduce our more
detailed version, called the {\bf full quiver of a
representation}, whose description relies on Frobenius gluing, and
give some examples of how one builds the full quiver from a
specific algebra.

In the subsequent sections we study full quivers in increasing
complexity, starting with single vertices and then single arrows
(which we call {\it elementary quivers}) and next {\it branches},
with the aim of interpreting properties of \Zcd\ algebras in terms
of their full quivers of representations. In the subsequent paper,
\cite{BRV3}, we distinguish among different kinds of full quivers
by means of PIs.

\subsection{Review of classical quivers}

\begin{defn}\label{quiv}
Suppose $\{ e_1, \dots, e_t\}$ is a 1-sum set of orthogonal
primitive idempotents of a finite dimensional algebra $A$ over a
field $F$. The {\bf (classical) quiver} of $A$ is the graph whose
vertices are $\{ e_1, \dots, e_t\}$ and whose arrows from $e_i$ to
$e_j$ have
multiplicities $\dimcol{e_i\Rad(A)e_j}{F}$.
\end{defn}

The quiver is independent of the choice of $\{ e_1, \dots, e_t\}$,
in view of the well-known identification of $eAe$ with $\Hom_A(Ae,
Ae)$ and the Krull-Schmidt Theorem.

\begin{exmpl} (i) $A = F[x]/\ideal{x^2}$ The quiver is
the loop:

\begin{equation}\label{loop}
\xymatrix@C=40pt{
\xy
(0,0)?(0)*{\bullet};(0,0) **\crv{(-10,5) & (-10,-5)}?(0.99)*\dir{>};
\endxy
}
\end{equation}
\smallskip

(ii) $A = F[x]/\ideal{x^3}$. The quiver is the double loop:

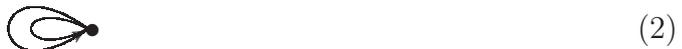
\begin{equation}
\xymatrix{ \xy
(10,0);(10,0) **\crv{(0,5) & (0,-5)}?(0.99)*\dir{>};
(10,0);(10,0) **\crv{(-4,10) & (-4,-10)}?(0.99)*\dir{>};
\endxy
}\!\!\bullet
\end{equation}

(iii) $A = \set{\(\begin{array}{cccc}
* & * \\
0 & *
\end{array}\)} \subset \M[2](F)$. The quiver is the graph comprised of a single
arrow connecting two vertices:
\begin{equation*}
\xymatrix{ \xy
(0,0)?(0)*{\bullet};(10,0) **\crv{-}?(0.96)*\dir{>} ?(1)*{\bullet};
\endxy
}
\end{equation*}
\end{exmpl}

When considering quivers, one usually passes to $A/J^2$, where $J$
is the radical of~$A$; then there is the celebrated correspondence
between indecomposable modules and Dynkin diagrams, \cite {BGM},
\cite {Ga},
 \cite {Kac1}, \cite {Kac2}. Of course, in passing to
$A/J^2$ one forfeits much of the combinatoric structure involving
the radical.

Also, one customarily considers the ``basic algebra''
$\operatorname{End}_A P$ Morita equivalent to $A$, for a suitable
projective module $P$. These reduced classical quivers, although a
powerful tool in much of representation theory, do not help us
very much in combinatoric questions such as determining PIs of an
algebra.

\subsection{Definition of the full quiver}

Having established the notion of \Zcd\ algebras, we would like to
study their structure through the techniques of representations of
finite dimensional algebras, by refining the notion of quiver from
Definition~\ref{quiv}. Since we want to study \Zcd\ algebras, and
these are defined in terms of representations, we need our notion
of quiver to be compatible with the particular representation.

Since \Zcd\ algebras are semilocal, i.e., $S = A/J$ is semisimple
and $J$ is nilpotent, let us begin in the somewhat greater
generality of semilocal rings. Writing $S = \prod _{i=1}^k
\M[n_i](D_i),$ where $D_i$ are division rings with multiplicative
unit $\bar e_i$, we can lift the orthogonal idempotents $\bar e_1,
\dots, \bar e_k$ to a set of orthogonal idempotents $e_1, \dots,
e_k$ of~$A$ whose sum is 1.

The fact that $A = S \oplus J$ enables us to start with the quiver
of $S$ as a sub-quiver of that of $A$. We treat these parts in
different ways. The loops of the classical quiver could be viewed
as ring injections inside $S$, which in turn could be broken down
into homomorphisms $D_i \to D_{i'},$ and the arrows that are not
loops are separated to arrows each of multiplicity 1.

Unfortunately, it can be quite difficult to describe all
homomorphisms between two given division rings. But in the case
under consideration, of a Zariski closed $F$-subalgebra~$A$ of
$\M[n](K)$, the center of each $D_i$ is either a finite field or
an algebraically closed field, so each $D_i$ is a field. Thus, our
injections are either the identity map or are given by powers of
Frobenius homomorphisms $a \mapsto a^p$. In light of \Dref{Wbf},
which involves partitioning the blocks into glued components, our
representation is determined by the following:

\begin{defn}\label{quivers}
Suppose we are given a representation $\rho \co A \to \M[n](K)$,
where $A$ is Zariski closed and $\rho(A)$ is in Wedderburn block
form, with the blocks $B_1,\dots,B_\ell$ along the diagonal.

\begin{enumerate}
\item The {\bf full quiver of the representation} $\rho$ (of $A$)
is a directed graph on $k$ vertices corresponding to the diagonal
blocks $B_r$, $1 \le r \le \ell$. Each vertex is labelled with a
roman numeral {\bf i} (taking the values $I$, $\II$ etc.); vertices corresponding to glued
diagonal blocks are labelled with the same roman numeral. The
label $\,\bullet\,$ denotes a vertex which is not glued. In an
algebra without unit, the label $\,\circ\,$ indicates a zero
diagonal
block.
\item Each vertex is assigned a pair of subscripts $(n_r,t_r)$,
written as $\comp{n_r}{t_r}$, where $\card{F_r} = \card{F}^{t_r}$.
The first subscript $n_r$ is the degree of the corresponding
matrix algebra $B_r = \M[n_r](F_r)$, and is called the {\bf matrix
degree} of the vertex; $\sum_{r=1}^{k} n_r = n$. We may suppress
the first subscript when $n_r = 1$.
\item When $F_r = \cent(B_r)$ is a finite field, the second
subscript $t_r$ denotes that the cardinality of $F_r$ is
$\card{F}^{t_r}$; i.e., $[F_r:F] = t_r$. In this case the vertex
is called {\bf finite}. Otherwise, the second subscript is left
out, which indicates that   $F_r = K$; then the vertex is called
{\bf infinite}. (When $A$ is Zariski closed, $F_r$ is either
finite or equal to $K$ as explained above).

We simplify the notation a bit, by means of the observation that
glued diagonal blocks must have the same pairs of subscripts, so
we only notate this for the first block in a given glued diagonal
component.

\item When $F$ is a finite field, say of order $q$, superscripts
$u \le t$ indicate the power $\phi^u$ of the Frobenius
homomorphism $\phi \co x \mapsto x^q$; the notation $I^{(u)}$ and
$I^{(\ell')}$ indicates blocks belonging to the same glued
component, with the identification given by $\phi^{u' - u}$.
The index $u' - u$ is well-defined modulo the corresponding
dimension $t_r$ of $F_r$ over $F$. Absence of superscripts should
be read as $u = u'= 0$; namely, identity gluing.
\item The full quiver has an arrow (i.e., a directed edge) from {\bf i}$_r$ to
{\bf i}$_s$ iff $e_r\rho(A) e_s\ne 0$, where $e_r,e_s$ are the idempotents corresponding to the vertices. Loops are omitted. Arrows with
identical gluing are labelled with the same Greek letter.

\item\label{fb} Proportional off-diagonal gluing  (in other words,
the entries in $B_{r',s'}$ each are $\nu$ times the respective
entries in $B_{r,s}$) is indicated by writing some Greek letter
(say $\a$) for the arrow from the $r$ vertex to the $s$ vertex,
and writing $\nu \a$ for the arrow from the $r'$ vertex to the
$s'$ vertex.

Proportional Frobenius gluing is indicated by writing
$\a^{q^{u}}$ for the arrow from the $r$ vertex to the $s$
vertex, and writing $\nu \a^{q^{u'}}$ for the arrow from the
$r'$ vertex to the $s'$ vertex. (We do not write $\nu$ when $\nu =
1$.)
\end{enumerate}
\end{defn}

Since the representation is upper diagonal, there can be no
cycles.  In \Tref{propglu} we show when $F$ is infinite that the
representation can be chosen so that its full quiver has only
proportional gluing.

\begin{rem}
Conversely, there is an algebra (provided with a representation) corresponding to
any full quiver, which can be constructed using path algebras.
\end{rem}

\begin{rem}\label{3.7}
Any gluing applies in the same way to each entry in the matrix
components (and thus corresponds to a direct sum of projective
modules). Thus, shrinking the size of all blocks in a glued
component to $1$
yields a Morita equivalent algebra.
\end{rem}

\begin{rem}\label{irrc}
For any algebra $A = A_1 \times \cdots \times A_t$, taking the
representation of $A$ that is the direct sum of the
representations of the $A_i,$ we see that the full quiver of this
representation is the disjoint union of the full quivers of the
representations of the $A_i$.

We call an algebra \textbf{indecomposable} if it cannot be written
as a direct sum of proper subalgebras. Thus, when there is no
gluing, the full quiver of a representation is connected (as an
undirected graph) iff the corresponding algebra is indecomposable.

We say that two connected components of the full quiver are {\bf glued} if they have some glued vertices. For example, the full quiver
$$I \ra \II, \ \II \ra \III$$
is disconnected, but its components are glued by the vertices
labeled $\II$. Extending this relation via transitivity, we see
that any representation factors into a direct product of its
glue-connected components; consequently, we focus on
glue-connected full quivers.
\end{rem}

\begin{exmpl}
An algebra may have different representations in Wedderburn block form, of the same dimension, which are not conjugate to each other. At times, this phenomenon is related to arrows between glued vertices, as illustrated by the two representations below.
\begin{equation*}
\xymatrix@C=32pt@R=2pt{
\comp[I]{1}{K} \ar[r] \ar[rdd] & \comp[\II]{1}{K}
\\
{} & {} & {}
\\
\comp[I]{1}{K} & \comp[\II]{1}{K} }
\begin{array}{c} {\ } \\ {\ } \\ \mbox{{;}} \end{array}
\qquad \qquad
\xymatrix@C=40pt@R=4pt{
\comp[I]{1}{K} \ar[r] & \comp[\II]{1}{K}
\\
{} & {} & {}
\\
\comp[I]{1}{K} \ar[r] & \comp[\II]{1}{K} }
\end{equation*}
\end{exmpl}

\begin{prop} If $\Gamma$ is the full quiver of a representation of an
algebra $A$, then the full quiver $\Gamma'$ obtained by reversing
all arrows in $\Gamma$ is the full quiver of the representation of
the opposite algebra ${A}^{\operatorname{op}}$.
\end{prop}
\begin{proof}
$\Gamma'$ could be viewed as the full quiver of the transpose of
the representation (where we permit blocks beneath the diagonal,
rather than above it), and this is a transpose representation of
$A$, which is isomorphic to the opposite algebra of $A$.
\end{proof}

\subsection{Primitive and imprimitive arrows}

We streamline the full quiver a bit, to ease the notation.

\begin{defn}
An arrow ${\bf i}\ra {\bf j}$ is called {\bf primitive} if there
is no $\bf k$ such that there exist arrows ${\bf i} \ra {\bf k}$
and ${\bf k} \ra {\bf j}$. We erase all non-primitive arrows
involved in unglued vertices, since these are formal consequences
of the primitive arrows, corresponding to the product of the
corresponding respective elements of the algebra. Thus, we write
$$I \to \II \to \III \qquad \mbox{or} \qquad \bullet \to\bullet \to \bullet$$
for the full quiver of the algebra of upper triangular $3 \times
3$ matrices, although technically there should be an arrow from
$I$ to $\III$. (Gluing of arrows can upset this reasoning, as
indicated in Example~\ref{gluerad} below; therefore glued
non-primitive arrows will not be erased).
\end{defn}

\begin{exmpl}\label{ex0}
\begin{enumerate}
\item\label{EX1}
The algebra
$$A = \set{\(\begin{array}{cccc}
\alpha & * & * & *\\
0 & \beta & \gamma &* \\
0 & 0 & \alpha &* \\
0&0&0& \beta ^q
\end{array}\) \suchthat \alpha \in
\M[2](\F_{q^3}),\ \beta \in K,\ \gamma \in \M[1,2](K)}$$ over $F =
\F_q$, where $K \supseteq \F_{q^3},$ has the following full
quiver:

\vskip 0.3cm
\begin{equation}\label{ex8}
\xymatrix@C=40pt{
\comp[I]{2}{3} \ar@/^0pt/[r] \ar@/^16pt/[rr] \ar@/^26pt/[rrr]
&
\compfull[\II]{1}{K}\ar@/_12pt/[rr] \ar@/^0pt/[r]|(0.5){\gamma }
&
\comp[I]{2}{3} \ar@/^0pt/[r]
&
\II^{(1)} }
\end{equation}
\vskip 0.6cm
or just the path
\begin{equation}\label{ex8-1}
\xymatrix@C=40pt{
\comp[I]{2}{3} \ar@/^0pt/[r]
&
\comp[\II]{1}{K} \ar@/^0pt/[r]
&
I \ar@/^0pt/[r]
&
\II^{(1)} },
\end{equation}
in view of our short-hand conventions and erasing the
non-primitive arrows. The representation written out in full, in
$\M[6](K)$, is
$$\set{ \(\begin{array}{cccccc}
\alpha _{11} & \alpha _{12} & * & * & * &*\\
\alpha _{21} & \alpha _{22} & * & * & * &*\\ 0 &
0 & \beta & \gamma_1 & \gamma _2 &* \\
0 & 0 & 0&\alpha _{11} & \alpha _{12} & * \\
0 & 0 & 0&\alpha _{11} & \alpha _{12} & * \\
0&0&0&0 &0& \beta ^q
\end{array}\) \suchthat \alpha_{i,j}, \beta, \gamma _i \in K}.$$

\item In contrast, the full quiver
\begin{equation}\label{ex6}
\xymatrix@C=40pt{
\comp[I]{2}{3} \ar@/^0pt/[r] \ar@/^12pt/[rr]
&
\II \ar@/_12pt/[rr]
&
I \ar@/^0pt/[r]
&
\II^{(1)} } \\
\end{equation}
represents the algebra having all the relations of $A$, as well as
the extra relation $\gamma_1 = \gamma_2 = 0$, namely $$
\set{\(\begin{array}{cccccc}
\alpha _{11} & \alpha _{12} & * & * & * &*\\
\alpha _{21} & \alpha _{22} & * & * & * &*\\ 0 &
0 & \beta & 0 & 0 &* \\
0 & 0 & 0&\alpha _{11} & \alpha _{12} & * \\
0 & 0 & 0&\alpha _{11} & \alpha _{12} & * \\
0&0&0&0 &0& \beta ^q
\end{array}\) \suchthat \alpha_{i,j}, \beta \in K}.$$
(This is seen most easily by erasing the middle straight arrow in
\eq{ex8}; now two of the curved arrows become primitive, so must
be retained on notating the full quiver.)
\end{enumerate}
\end{exmpl}

\subsection{Proportional Frobenius gluing (of arrows)}

Like vertices, arrows correspond to matrix blocks defined over
$K$, and are said to be of \textbf{finite type} if they can be
defined over a finite subfield of $K$; otherwise they are said to
be of \textbf{infinite type}. As with vertices, the complications
in gluing arise for arrows of finite type.

\begin{prop}\label{p2p}
If $\Gamma$ is the full quiver of a representation of an algebra
$A$, and if two primitive arrows in $\Gamma$ are proportionally
Frobenius glued with respect to $q^{\ell},q^{\ell'}$ in
\Dref{quivers}.\eq{fb}, then their respective vertices are glued
with the same Frobenius twist.
\end{prop}
\begin{proof}
Suppose that for some constant $\nu \in K$ we have $b_{r',s'} =
\nu b_{r'',s''}$ for every $a= \sum _{r,s} b_{r,s}\in A$, where
$b_{r,s} \in B_{r,s}$. Since the diagonal blocks of the
representation each are simple algebras, we may assume that
$b_{r,r} \in F_r$. If the algebra $A$ contains an element $\hat a
$
 whose diagonal component is $\sum \hat b _r e_{r}$,
again assuming that $\hat b_r \in F_r$, we can normalize to assume
$b_{r',r'} = \hat b_{r'}$. But $A$ also contains $\hat a a = \sum
_{r,s} {\hat b}_r b_{r,s}$, so taking $ \hat a a - a^2$, we note
that the $(r',s')$ coefficient is 0 whereas the $(r'',s'')$
coefficient is nonzero, and we have ``unglued'' the $r'',s''$
block from the $r',s'$ block unless $\hat b_{r''} =  b_{r'}$.

 The same
argument holds for the second entry, when we fix $r$ and let $s$
vary.
\end{proof}

\begin{rem}\label{unglu} The same argument of ``ungluing'' yields somewhat more precise information. Namely, notation as  in Remark~\ref{offdiag}, shows
that the block corresponding to the primitive arrow must belong to
a field extension $L$ of the field of order $q ^{\gcd(t_r,t_s)}$,
hence $|L| = q ^{c \cdot \gcd(t_r,t_s)}$ for some $c$ or $L=K$ is
infinite.

 Furthermore, the
Frobenius twist of the arrows must equal $q^{(\ell'- \ell) +
m\cdot{\gcd(t_r,t_s)}}$ for some~$m$, since otherwise multiplying
by a diagonal component would ``unglue'' the arrows.
\end{rem}

\subsection{Synopsis of the full quiver}

To summarize, the full quiver of $A \sub \M[n](K)$ is a graph
without cycles which displays the following information:
\begin{enumerate}
\item A vertex for
each of the $\ell$ diagonal blocks, recording diagonal gluing; \item For each
glued component ${\bf i}$: the size $n_{\bf i}$ of its blocks, and
the order $q^{u_{\bf i}}$ of the underlying field (which is $K$
or a finite extension of $F$, when the latter is finite); \item
The identifications arising from Frobenius gluing along the
diagonal;
\item An arrow (directed edge) $r\ra s$ connecting the $r,s$
vertices, for $r < s$, iff the block $ B_{r,s} \ne 0$. Only
primitive arrows are notated, as well as glued imprimitive arrows;
\item The identifications arising from proportional Frobenius gluing above the diagonal.
\end{enumerate}

\subsection{The full quiver covers the classical quiver}

Let us see how the full quiver covers the classical quiver, while
providing more combinatoric information about the algebra.

\begin{rem}\label{J2}
One can interpret the components of the full quiver in terms of
the structure of the corresponding algebra.   Primitive arrows
corresponds to components of $J\setminus J^2$, in analogy to the
classical quiver.
\end{rem}

\begin{rem}
For any algebra $A$ over an infinite field $F$, the classical
quiver of $A$ is the homomorphic image (as a graph) of the full
quiver of any faithful representation of~$A$. (When the base field
$F$ is finite, this assertion may fail for two reasons: We may
have non-identical Frobenius gluing, and the classical quiver need
not be finite, because of infinite dimensionality of the center of
a block over $F$; see \Eref{compress}).
\end{rem}

The main difference from classical quivers, and the reason that we
use full quivers of representations, is that by opening up glued
components, each full quiver is a directed tree (i.e., has no
directed cycle), and has no multiple arrows. This is clear from
the definition, and will be illustrated in our examples.

\begin{exmpl}
The algebra $A_1 = \set{\(\begin{array}{cccc}
\alpha & * & * & *\\
0 & \beta & * &* \\
0 & 0 & \alpha &* \\
0&0&0& \beta^q
\end{array}\) \suchthat \alpha \in
\F_{q^3}, \ \beta \in K}$ taken over $F = \F_q$ (where $K =
\bar{F}$), which is Morita equivalent to \Eref{ex0}.\eq{EX1}, has
full quiver notated as
$$\xymatrix{\comp[I]{1}{3} \ar[r] &  \II \ar[r] &  I \ar[r] & \II^{(1)}}.$$

If instead we took $K A_1 = \set{\(\begin{array}{cccc}
\alpha & * & * & *\\
0 & \beta & * &* \\
0 & 0 & \alpha &* \\
0&0&0& \beta ^q
\end{array}\) \suchthat \alpha ,
 \beta \in K}$, as an algebra over $K$,
then our full quiver would be notated as
$$\xymatrix{I \ar[r] &  \II \ar[r] &  I \ar[r] &  \II^{(1)}}.$$
The classical quiver for this algebra is

\begin{equation}\label{ex0-}
\xymatrix@C=40pt{
I \ar@(dl,ul) \ar@/^6pt/[r] \ar@/^10pt/[r] \ar@/^14pt/[r]
&
\II \ar@(dr,ur) \ar@/^10pt/[l]
&
},
\end{equation}
\smallskip

 \noindent which we get by making the appropriate
identifications in the full quiver (without first suppressing
imprimitive arrows).
\end{exmpl}

 \medskip
\begin{exmpl}
$A = \set{\(\begin{array}{cccc}
\alpha & * & * & *\\
0 & \beta & 0 &* \\
0 & 0 & \beta &* \\
0&0&0&\alpha
\end{array}\) \suchthat \alpha ,
 \beta \in K}$, whose full quiver is notated as
\smallskip
\begin{equation}
\xymatrix@C=40pt{
I \ar@/^0pt/[r] \ar@/^10pt/[rr]
&
\II \ar@/_10pt/[rr]
&
\II \ar@/^0pt/[r]
&
I}.
\end{equation}
The classical quiver for this algebra is
\smallskip
\begin{equation}\label{ex0-+}
\xymatrix@C=40pt{
I \ar@(dl,ul) \ar@/^6pt/[r] \ar@/^10pt/[r]
&
\II \ar@/^10pt/[l] \ar@/^6pt/[l]
},
\end{equation}
\smallskip
which again we get by making the appropriate identifications in
the full quiver (without first suppressing imprimitive arrows).
\end{exmpl}

\subsection{Sub-quivers}

\begin{defn} \label{quivers1}
A {\bf morphism} $\psi\co \Gamma _1 \to \Gamma _2$ of full quivers
is a morphism of directed graphs (i.e., if ${\bf i} \ra {\bf j}$
is an arrow, then $\psi({\bf i}) \ra \psi({\bf j})$ is also an
arrow), preserving gluing, such that if $n_{\bf i}(t_{\bf i})$ is
the subscript for the vertex $\bf i,$ then $\M[n_i](\F_{q^{t_i}})$
embeds into $\M[n_{\psi (\bf i)}](\F_{q^{t_{\psi(\bf i)}}})$,
whereby the superscripts are carried with their respective
vertices.

We say that $\Gamma_1$ is a {\bf sub-quiver} of $\Gamma_2$ if
there is a 1:1 morphism from $\Gamma _1$ to $\Gamma_2$.

The {\bf infinite sub-quiver} of a full quiver is the sub-quiver
comprised of infinite vertices (i.e., vertices $e_r$ for which the
corresponding field $F_r$ is infinite) together with the arrows
between them.
\end{defn}

Note that a morphism could send a primitive arrow to an
imprimitive arrow,
\begin{exmpl}
\begin{enumerate}
\item $\comp[I]{3}{2} \to \comp[\II]{1}{5} \to \comp[\III]{1}{3}$
is a sub-quiver of $$\comp[I]{1}{K} \to \comp[\II]{4}{4} \to
\comp[\III]{1}{5} \to \IV \to \comp[V]{1}{6},$$ under the morphism
sending $I \mapsto \II$, $\II \mapsto \III$, and $\III \mapsto V$,
but is not a sub-quiver of $I \to \comp[\II]{4}{4} \to
\comp[\III]{1}{5} \to \IV \to \comp[V]{1}{2}$. \item $\comp{1}{4}$
is a sub-quiver of $\comp{2}{2}$ because of the regular
representation $\F_{q^4} \sub \M[2](\F_{q^2})$.
\end{enumerate}
\end{exmpl}

\subsection{Full quivers viewed structurally}

\begin{rem}\label{Zariskiclq} Note that the quivers come from
 components of the affine variety defined by the \Zcd\ algebra,
 which can be described in terms of Remark~\ref{decomp}.
The vertices come from the semisimple components, which are just
matrix algebras. However, the arrows are somewhat more
complicated.

In principle, the arrows correspond to matrix units arising in the
representation corresponding to the given Wedderburn block form.
When the base field $F$ is infinite, then $K=F$, and the \Zcr\ is
just a finite dimensional algebra; in this case, any arrow can be
identified with the appropriate element in the algebra.  We call
this the element of $A$ {\bf corresponding to} the arrow.

But when $F$ is finite, the situation is subtler. We could have a
situation such as
$$A = \set{\(\begin{array}{cccc}
F & K \\
0 & K
\end{array}\)} \subset \M[2](K),$$
whose full quiver is
$$\xymatrix@C=40pt{
\comp[\bullet]{1}{1} \ar@<2pt>@/^0pt/[r] & \comp[\bullet]{1}{K}}
$$

Note that when $F_1 \subset F_2 \subset K$ we get the same full
quivers $$\xymatrix@C=40pt{
\comp[\bullet]{1}{1} \ar@<2pt>@/^0pt/[r] & \comp[\bullet]{1}{2}}
$$from the algebras $\set{\(\begin{array}{cccc}
F_1 & F_2 \\
0 & F_2
\end{array}\)} $ and $\set{\(\begin{array}{cccc}
F_1 & K \\
0 & F_2
\end{array}\)} $, so
 we cannot recover the radical component from the arrow.

In this case, we view the arrows  as corresponding to linear
functionals, and we often are led to use the generic elements
obtained in \cite[Construction 7.14]{BRV1}. By definition, since
any generic element is taken from a given component in the affine
variety defined by the Zariski closure, the choice of generic
element for this component does not affect the polynomial
relations. Thus the gluing does not depend on the particular
choice of generic element corresponding to a given arrow.
\end{rem}

Here is a more precise way of viewing the last paragraph of the
previous remark.

Let us record a monoid action needed for subsequent papers.

\begin{rem}
\label{grade} Let $\mathcal M_i$ denote the multiplicative monoid
$\mathbb Z/q^{t_i} \mathbb Z \cup \{ 0 \}.$ Then  $\mathcal M_k$
acts naturally on $F_k$ via $ [j] a = a^j,$ where $a^0 = 1.$
Consequently, $\mathcal M_1 \times  \dots \times \mathcal M_k$
acts on $A$, leading to a grading in terms of the sub-Peirce
decomposition and Frobenius gluing. In \cite{BRV4} we use this
grading to study the polynomial identities of $A$.
\end{rem}

\section{Specific classes of full quivers}

In this section, we describe various classes of full quivers,
starting with the easiest case, and then increasing in complexity.
A full quiver is called a {\bf path} if its initial and terminal
vertices each have degree $1$ and all other vertices have degree
$2$. Let us consider some examples of paths and other relatively simple examples of full quivers.

\subsection{Full quivers of semisimple algebras}

Having described representations of semisimple algebras in Remark
\ref{semis1}, we can easily give their full quivers, which are
just finite sets of disconnected vertices (since the radical $J =
0$). The full quiver of a matrix algebra is a single bullet
$\bullet$ without any arrows. The full quiver of a direct sum of
matrix algebras is just a set of  disconnected bullets, in view of
Remark \ref{irrc}.

\subsection{Elementary full quivers}

After the case of unconnected vertices without arrows, the next
most basic kind of full quiver, called {\bf elementary}, is a full
quiver that consists of a single arrow; this corresponds to module
and bimodule structures over a semisimple algebra. Any algebra
having an elementary full quiver satisfies $J^2 = 0$; let us
consider the various possibilities.

\begin{enumerate}
\item\label{E1} No gluing. The full quiver $\comp[I]{m}{t_1} \to
\comp[\II]{n}{t_2}$ corresponds to the algebra $$A =
\smat{\M[m](F_1)}{L}{0}{\M[n](F_2)},$$ where $F_i$ are
intermediate subfields of $F\sub K$ of order $|F|^{t_i}$ (or are
infinite in the absence of the $t_i$), and $L$ is a non-zero
$\M[m](F_1),\M[n](F_2)$-subbimodule of the upper right part of
$\M[m+n](K)$. The defining relations are $\lam_{m+i',j} = 0$,
$\lam_{ij}^{\card{F_1}} = \lam_{ij}$, and
$\lam_{m+i',m+j'}^{\card{F_{2}}} = \lam_{m+i',m+j'}$, for all $1
\le i, j \le m$, $1 \leq i', j' \leq n$ (indeed, these relations
say that the first diagonal block has coefficients in $F_1$ and
the second diagonal block has coefficients in $F_2$), and the
extra relations defining $L$.

\item\label{E2} Identical gluing. The full quiver $\comp[I]{n}{t}
\to I$ corresponds to the algebra $A$ as in~\eq{E1}, but with $m =
n$, and with gluing. Note that $A \sub \M[n](K[\la]/\ideal{\la
^2})$. (If $F$ is infinite, then  $\id(A) = \id(\M[n](F).$ If
$|F_1| = q,$ and $g$ is a central polynomial for $M_n(F_1),$ then
$g^q -g \in \id(M_n(F_1))$ but is not in $\id(A)$.)

For $m=n=1$ we have $A =
\set{\smat{\alpha}{\beta}{0}{\alpha}\suchthat \alpha \in F_1,\,
\beta \in L}$ where $F_1$ is as before, and $F_1 \sub L \sub K$ is
any (\Zcd) intermediate field. We have the same relations in
\eq{E1}, plus $\lam_{22} = \lam_{11}$.

\item\label{E3} Frobenius gluing.
The full quiver $\comp[I]{n_1}{t_1} \to I^{(u)}$ corresponds to
the algebra analogous to \eq{E2}, but with Frobenius gluing of
degree $\ell$. For $m=n=1$ we have $$A =
\set{\smat{\alpha}{\beta}{0}{\alpha^{\pi}}\suchthat \alpha \in
F_1,\, \beta \in L}$$ where $F_1,L$ are as before, and $\pi =
q^u$. We have the same relations as for \eq{E2}, with
$\lam_{22} = \lam_{11}^{\pi}$ replacing $\lam_{22} = \lam_{11}$.
Note that
$$\smat{\alpha}{\beta}{0}{\alpha^{\pi}}
\smat{\alpha'}{\beta'}{0}{{\alpha'}^{\pi}} =
\smat{\alpha\alpha'}{\alpha\beta'+\beta
{\alpha'}^{\pi}}{0}{\alpha^{\pi}{\alpha'}^{\pi}},$$ so in some
sense the multiplication on the 1,2 component is ``skewed'' by a
$\pi$-power. For general $m$ and $n$ we take $\pi$-powers of the
corresponding coordinates.

\item\label{E4} $\circ \to \bullet$ corresponds to the algebra
$\set{\smat0*0*}$ without 1. Similarly, $\bullet \to \circ$
corresponds to the algebra $\set{\smat**00}$ without 1. (These are
similar to \eq{E1}, where one of the diagonal components is 0.
Here, $\id(A)$ is $x\id(M_n(F_1))$ and $\id(M_n(F_1))x$,
respectively, by Lewin's Theorem.)

\item\label{E5} $\circ \to \circ$ corresponds to the set of
matrices whose only nonzero entry is in the upper right corner of
$2 \times 2$ matrices. (Here, $\id(A) = xy.$)
\end{enumerate}

\subsection{Branches of full quivers}

In order to investigate full quivers further, we need the following important notion.

\begin{defn}
A path in a full quiver is a {\bf branch} if it cannot be
extended, namely if no arrow enters its first vertex, and no arrow
exits its final vertex.
The {\bf length} of a path is its number of arrows, which is one
less than the number of vertices. (Thus a branch of length $\ell
-1$ has $\ell$ vertices.) Likewise, an arrow has {\bf length}
$\ell$ if there is a path from $i$ to $j$ having length $\ell$.
\end{defn}

\begin{rem}
Since radical elements correspond to arrows, we see that the index
of nilpotence of an algebra is at most $\ell +1$, where $\ell$ is
the maximal length of the paths of its full quiver.
\end{rem}

\begin{rem}\label{trtwist}
Given a single branch with Frobenius proportional gluing, one can
make all the parameters equal to $1$ via a change of base. Thus,
in this case, pure proportional gluing in any given branch becomes
identical gluing.
\end{rem}

\section[Gluing, compression and change of base]{Gluing, compression of
 full quivers, and change of base ring}

The definition of quivers is combinatoric. Although we have
defined full quivers of linear representations over a field, the
same definition applies to linear representations over any
commutative affine algebra, such that all entries are given in
blocks along or above the diagonal.

Closely related representable algebras (involving
 gluing) could require (faithful) representations of widely
varying degrees, which we can describe better in terms of
representations of much lower degree over arbitrary commutative
affine algebras. Let us first describe commutative affine algebras
in terms of full quivers of representations. We start with unglued
arrows.

\begin{prop}\label{comm2}
An indecomposable algebra (see \Rref{irrc}) with unglued arrows is
commutative iff all of its vertices are identically glued and
every branch has length $\le 1$.
\end{prop}
\begin{proof} $(\Rightarrow)$ If we
have an arrow between unglued vertices, then the algebra cannot be
commutative, since $e_{ii}e_{ij} = e_{ij}$ whereas $e_{ij}e_{ii} =
0$. In light of \Rref{irrc}, this means all vertices must be
glued; the gluing must all be identical, because $\alpha e_{ii} +
\alpha^s e_{jj}$ does not commute with $e_{ij}$ for $\alpha^s \neq
\alpha$. Moreover, the same argument ($e_{ij}e_{jk} = e_{ik}$
whereas $e_{jk}e_{ij} = 0$) shows that there cannot be two
consecutive arrows.

$(\Leftarrow)$ is obvious.
\end{proof}

\begin{exmpl}\label{5.2}
The full quivers of \Pref{comm2} are bi-partite, with disjoint
sets of
sources and sinks. For example,
\vskip-0.6cm
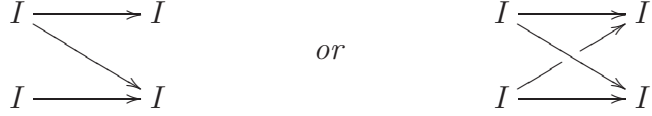
\begin{equation*}
\xymatrix@C=40pt@R=4pt{
\comp[I]{1}{K} \ar[r] \ar[rdd] & \comp[I]{1}{K}
\\
{} & {} & {}
\\
\comp[I]{1}{K} \ar[r] & \comp[I]{1}{K} }
\begin{array}{c} {\ } \\ {\ } \\ \mbox{or} \end{array}
\qquad \qquad
\xymatrix@C=40pt@R=4pt{
\comp[I]{1}{K} \ar[r] \ar[rdd] & \comp[I]{1}{K}
\\
{} & {} & {}
\\
\comp[I]{1}{K} \ar[r] \ar[ruu]|(0.5){\hole} & \comp[I]{1}{K} }
\end{equation*}
\end{exmpl}

\subsection{Compressed full quivers}

When glued arrows are allowed, there are many other possibilities
for representations of commutative algebras.
\begin{exmpl}
The local commutative algebra $F[\lam]/\ideal{\lam^4}$ has the presentation
\vskip 0.2cm
\begin{equation}\label{exJ}
\xymatrix@C=40pt{
\comp[I]{1}{K} \ar@/^0pt/[r]|(0.5){\beta} \ar@/^14pt/[rr]|(0.5){\gamma}
&
\comp[I]{1}{K}
\ar@/^14pt/[rr]|(0.25){\hole}|(0.501){\gamma}
\ar@/^0pt/[r]|(0.4){\beta }
&
\comp[I]{1}{K} \ar@/^0pt/[r]|(0.4){\beta}
&
\comp[I]{1}{K}},
\end{equation}
in which the $(1,2)$, $(2,3)$ and $(3,4)$ entries are identically
glued, as well as the $(1,3)$ and $(2,4)$ entries. (The
non-primitive arrow of length $2$ is omitted from the diagram).
Omitting all the non-primitive arrows, as in
\begin{equation}\label{exJ2}
\xymatrix@C=40pt{
\comp[I]{1}{K} \ar@/^0pt/[r]|(0.5){\beta}
&
\comp[I]{1}{K}  \ar@/^0pt/[r]|(0.4){\beta }
&
\comp[I]{1}{K} \ar@/^0pt/[r]|(0.4){\beta}
&
\comp[I]{1}{K}},
\end{equation}
would release the gluing of length-1 arrows, yielding a
$4$-dimensional algebra which is not commutative. Frobenius gluing
of the length-1 arrows also yields non-commutative algebras.
\end{exmpl}

\begin{exmpl}\label{compress}
Generalizing the previous example, gluing within a branch leads to the following sort of situation:
\smallskip
$$B_\ell = \set{\(\begin{array}{cccccc}
\alpha & \beta & \gamma & \dots & \dots\\
0 & \alpha &\beta & \ddots & \vdots \\
0 & 0 & \ddots &\ddots &\gamma\\
0&0&0&\alpha & \beta \\
0&0&0&0 & \alpha
\end{array}\) \suchthat \alpha ,
 \beta, \gamma \in K}.$$
\smallskip

This algebra is isomorphic to $F[\lam]/\ideal{\lam^\ell}$, where
$\ell$ is the dimension of matrices. We call the full quiver of
this algebra the {\bf glued triangle} $I(\ell)$; in the previous
notation it is $I \to I \to \cdots \to I$, where all arrows of
equal length are identically glued. We call the glued triangle
$I(\ell)$ {\bf isolated} if the vertex $I$ does not appear
elsewhere in the branch. (Note for $|F|<\infty$ that
$\id(B_\ell)\supset \id(F).)$

(Allowing infinite dimensional representations, one can embed
$F[\lambda]$ in $\End(F^{\omega})$ by letting $\lambda$ be the
shift forward operator, corresponding to the infinite full quiver
$I(\infty) = I \to I \to \cdots$).

Conversely, given an isolated glued triangle (with identical
gluing) in a branch, we could compress it to a point (and erase
any duplication of arrows) and also erase all imprimitive arrows,
relabelling the vertex $\comp[\bf i]{d_{\bf i}}{n_{\bf i}}$ as
$\comp[\bf i]{d_{\bf i}}{n_{\bf i}}(\ell)$, to indicate that, in
that block, we have replaced the field $\F_{q_{\bf i}}$ by the
commutative affine algebra $\F_{q_{\bf i}}[\lam]/\ideal{\lam^\ell}
= \F_{q_{\bf i}}[\vareps_\ell],$ where ${q_{\bf i}} = q^{n_{\bf
i}}$ and $(\vareps_\ell)^\ell = 0$. The resulting full quiver
(after this process) is called the {\bf compressed quiver}.
\end{exmpl}

\begin{exmpl}
Compression of the glued triangle eliminates redundant,
identically glued, imprimitive arrows originating or ending in the
compressed idempotents. Therefore, one can compress
$$
\xymatrix@C=40pt{ \comp[I]{1}{K} \ar@/^0pt/[r]
\ar@/^11pt/[rr]|(0.5){\delta} & \comp[I]{1}{K}
\ar@/^0pt/[r]|(0.5){\delta} & \comp[\II]{1}{K}}$$ to
$$
\xymatrix@C=40pt{ \comp[I]{1}{K}(2) \ar@/^0pt/[r]
\ar@/^0pt/[r]|(0.5){\,\delta\,} & \comp[\II]{1}{K}},$$ and
\begin{equation*}
\xymatrix@C=40pt{
\comp[I]{1}{K}
    \ar@/^0pt/[r]
    \ar@/^14pt/[rr]|(0.5){\beta}
    \ar@/^26pt/[rrr]|(0.5){\alpha}
&
\comp[I]{1}{K}
\ar@/^14pt/[rr]|(0.25){\hole}|(0.501){\beta}
\ar@/^0pt/[r]|(0.5){\gamma}
&
\comp[\II]{1}{K} \ar@/^0pt/[r]
&
\comp[\II]{1}{K}}
\end{equation*}
to
\begin{equation*}
\xymatrix@C=40pt{
\comp[I]{1}{K}(2) \ar@/^0pt/[r]
&
\comp[\II]{1}{K}(2)};
\end{equation*}
but
$$\xymatrix@C=40pt{
\comp[I]{1}{K} \ar[r] & \comp[I]{1}{K} \ar[r] & \comp[\II]{1}{K}},
$$
which is a short-hand notation for
$$\xymatrix@C=40pt{
\comp[I]{1}{K} \ar@/^0pt/[r] \ar@/^11pt/[rr]|(0.5){\delta} &
\comp[I]{1}{K} \ar@/^0pt/[r]|(0.5){\epsilon} & \comp[\II]{1}{K}},
$$ cannot be compressed.
\end{exmpl}

Note that any morphism of full quivers naturally induces a
morphism of compressed quivers.

\begin{rem}
As noted in Example \ref{compress}, compression can enable us to
decrease the degree of a representation, at the expense of
introducing various commutative affine algebras as base rings for
the compressed vertices.

Compression of a branch involves adjoining elements $\vareps
_{i,\ell}$ to the base field $F$, for various $\ell,$ such that
$\vareps _{i,\ell}^\ell = 0$. Thinking of these $\vareps
_{i,\ell}$ as ``infinitesimals of order $\ell$,'' we see that the
compressed branch merely extends the base ring by adjoining
various infinitesimals. The down side of this approach is that the
algebraic interpretation of our full quivers becomes much more
complicated, since part of the radical now belongs to the
commutative base ring, and arrows no longer need correspond to 1:1
maps.
\end{rem}

Let us describe compression systematically, extracting the full
potential of identical gluing. A {\bf local component} of a given
vertex $I$ is a maximal path of the form $I(\ell)$ in which it
participates (a vertex can belong to more than one local
component, as seen in \Eref{5.2}). Let $P$ denote the set of
vertices glued to a given vertex $I$. If:
\begin{itemize}
\item[-] the local components of the vertices ${\bf i} \in P$ are
disjoint; \item[-] all the local components have the same length;
\item[-] arrows of the same length within all local components of
the same length are identically glued; \item[-] if ${\bf i}, {\bf
i}' \in P$ are in the same local component, and ${\bf j}, {\bf j}'
\in P$ are in the same local component, then the arrow from ${\bf
i}$ to ${\bf j}$ is identically glued to the arrow from ${\bf i}'$
to ${\bf j}'$; and \item[-] for every ${\bf j} \not \in P$, all
arrows from ${\bf i} \in P$ to ${\bf j}$ (when present) are
identically glued, and all arrows from ${\bf j}$ to ${\bf i} \in
P$ (when present) are identically glued,
\end{itemize}
then we may compress all local components, labelling each by
notation of the form $I(\ell)$. This process can be iterated,
labelling a subquiver of the form $I(\ell) \ra I(\ell)$ as
$I(\ell)(2)$ etc., until no further compression is possible; then
the full quiver is in {\bf incompressible form}. The name `local
component' remains appropriate, since the algebra
$R[\epsilon\subjectto \epsilon^{\ell} = 0]$ is local whenever $R$
is local. Since the full quiver is finite, we have
\begin{prop}
Every full quiver has a unique incompressible form.
\end{prop}

\begin{exmpl}\label{compressE}
\begin{enumerate}
\item The full quivers
$$
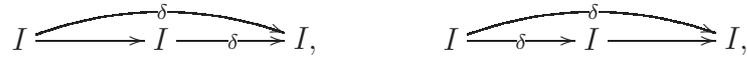
\xymatrix@C=40pt{ \comp[I]{1}{K} \ar@/^0pt/[r]
\ar@/^11pt/[rr]|(0.5){\delta} & \comp[I]{1}{K}
\ar@/^0pt/[r]|(0.5){\delta} & \comp[I]{1}{K},
&
\comp[I]{1}{K} \ar@/^0pt/[r]|(0.5){\delta}
\ar@/^11pt/[rr]|(0.5){\delta} & \comp[I]{1}{K} \ar[r] &
\comp[I]{1}{K},}
$$
can be compressed to
$$
\xymatrix@C=40pt{
\comp[I]{1}{K}(2) \ar@/^0pt/[r] \ar@/^0pt/[r] & \comp[I]{1}{K},
&
\comp[I]{1}{K} \ar[r] \ar@/^0pt/[r] & \comp[I]{1}{K}(2)},$$
respectively.

\item\label{Gl1} The full quiver
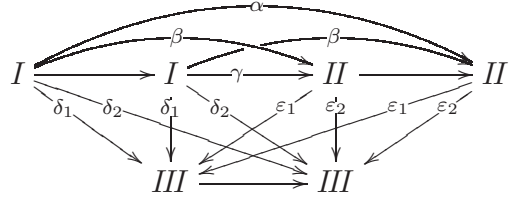
\begin{equation*}
\xymatrix@C=40pt{
\comp[I]{1}{K}
    \ar[rd]|(0.3){\delta_1} \ar[rrd]|(0.3){\delta_2}
    \ar@/^0pt/[r]
    \ar@/^14pt/[rr]|(0.5){\beta}
    \ar@/^26pt/[rrr]|(0.5){\alpha}
&
\comp[I]{1}{K}
    \ar[d]|(0.3){\delta_1} \ar[rd]|(0.3){\delta_2}
\ar@/^14pt/[rr]|(0.25){\hole}|(0.501){\beta}
\ar@/^0pt/[r]|(0.4){\gamma}
&
\comp[\II]{1}{K} \ar@/^0pt/[r]
    \ar[ld]|(0.3){\varepsilon_1} \ar[d]|(0.3){\varepsilon_2}
&
\comp[\II]{1}{K}
    \ar[lld]|(0.3){\varepsilon_1} \ar[ld]|(0.3){\varepsilon_2}
\\
{} & \comp[\III]{1}{K} \ar[r] & \comp[\III]{1}{K} & {} }
\end{equation*}
can be compressed to
$$\xymatrix@C=40pt{
\comp[I]{1}{K}(2) \ar[r] & \comp[\II]{1}{K}(2) \ar[r] &
\comp[\III]{1}{K}(2) },$$ with the quadruple arrow denoted by
$\alpha$ becoming imprimitive. \item \label{Gl2} Similarly,
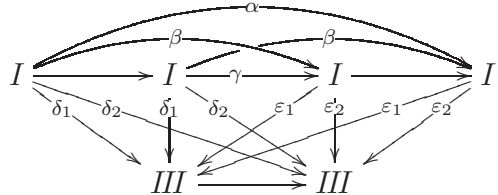
\begin{equation*}
\xymatrix@C=40pt{
\comp[I]{1}{K}
    \ar[rd]|(0.3){\delta_1} \ar[rrd]|(0.3){\delta_2}
    \ar@/^0pt/[r]
    \ar@/^14pt/[rr]|(0.5){\beta}
    \ar@/^26pt/[rrr]|(0.5){\alpha}
&
\comp[I]{1}{K}
    \ar[d]|(0.3){\delta_1} \ar[rd]|(0.3){\delta_2}
\ar@/^14pt/[rr]|(0.25){\hole}|(0.501){\beta}
\ar@/^0pt/[r]|(0.4){\gamma}
&
\comp[I]{1}{K} \ar@/^0pt/[r]
    \ar[ld]|(0.3){\varepsilon_1} \ar[d]|(0.3){\varepsilon_2}
&
\comp[I]{1}{K}
    \ar[lld]|(0.3){\varepsilon_1} \ar[ld]|(0.3){\varepsilon_2}
\\
{} & \comp[\III]{1}{K} \ar[r] & \comp[\III]{1}{K} & {} }
\end{equation*}can be compressed to
$$
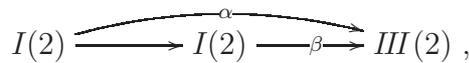
\xymatrix@C=40pt{
\comp[I]{1}{K}(2) \ar[r] \ar@/^12pt/[rr]|(0.5){\alpha} &
\comp[I]{1}{K}(2) \ar[r]|(0.5){\beta} & \comp[\III]{1}{K}(2) },$$
and further compressed to $\comp[I]{1}{K}(2)(2) \ra
\comp[\III]{1}{K}(2)$.
\end{enumerate}
\end{exmpl}
\subsection{Partial gluing up to infinitesimals}

Consider a fixed vertex $I$ in a full quiver, and let $A_0$ be the
subalgebra generated by all the sub-quivers of the form $I(\ell)$.
The matrix algebra represented by $I$ (which may itself be an
algebra over a local ring) is a homomorphic image of $A_0$, and
the kernel (generated by the arrows) is nilpotent. When there is
enough gluing so that this kernel is generated by one arrow, we
say that all the components of $A_0$ are {\bf glued up to
infinitesimals}.

\begin{exmpl}\label{gluebad}
Consider the following (incompressible) full quivers: \vskip 0.2cm
\begin{equation}\label{ex6.a}
\xymatrix@C=40pt{
I \ar@/^0pt/[r] \ar@/^14pt/[rr]|(0.5){\gamma}
&
I \ar@/^14pt/[rr]|(0.25){\hole}|(0.51){\gamma}
&
\II \ar[r]
&
\II } \\
\end{equation}
\begin{equation}\label{ex6.b}
\xymatrix@C=40pt{
I \ar@/^0pt/[r]|(0.5){\beta} \ar@/^14pt/[rr]|(0.5){\gamma}
&
I \ar@/^14pt/[rr]|(0.25){\hole}|(0.51){\gamma}
&
I \ar@/^0pt/[r]|(0.5){\beta}
&
I } \\
\end{equation}
\begin{equation}\label{ex6.c}
\xymatrix@C=40pt{
I \ar@/^0pt/[r]|(0.5){\beta} \ar@/^14pt/[rr]|(0.5){\gamma}
&
I \ar@/^14pt/[rr]|(0.25){\hole}|(0.51){\gamma}
&
I \ar@/^0pt/[r]|(0.5){\delta}
&
I } \\
\end{equation}

One can
view the algebra corresponding to diagram \eq{ex6.a} as
$$\set{\(\begin{array}{cccc}
\alpha & \delta & \gamma & \nu\\
0 & \alpha & 0 & \gamma \\
0 & 0 & \beta &\mu \\
0&0&0& \beta
\end{array}\) \suchthat \alpha, \beta, \gamma, \delta, \mu, \nu \in K}.$$

Partitioning these terms into $2\times 2$ matrices enables us to
rewrite this algebra as
$$\set{\(\begin{array}{cccc}
\phi & \chi \\
0 & \psi
\end{array}\) \suchthat \phi, \psi, \chi \in F[\vareps_2]},$$
where $\vareps_2$ satisfies the relations $\vareps_2^2 = 0$.

Likewise the algebra corresponding to diagram \eq{ex6.b} can be
viewed as
$$\set{\(\begin{array}{cccc}
\alpha & \delta & \gamma & \nu\\
0 & \alpha & 0 & \gamma \\
0 & 0 &\alpha &\delta \\
0&0&0& \alpha
\end{array}\) \suchthat \alpha, \beta, \gamma, \delta , \nu \in K},$$
which can be rewritten as
$$\set{\(\begin{array}{cccc}
\phi & \chi \\
0 & \phi
\end{array}\) \suchthat \phi, \chi \in F[\vareps_2]}.$$

The algebra corresponding to diagram \eq{ex6.c},
$$\set{\(\begin{array}{cccc}
\alpha & \beta & \gamma & \nu\\
0 & \alpha & 0 & \gamma \\
0 & 0 & \alpha &\delta \\
0&0&0& \alpha
\end{array}\) \suchthat \alpha, \beta, \gamma, \delta, \nu \in K},$$
can be presented with an extra relation, as
$$\set{\(\begin{array}{cccc}
\phi & \chi \\
0 & \psi
\end{array}\) \suchthat \phi, \psi, \chi \in F[\vareps_2], (\phi-\psi)^2 = 0}.$$
\end{exmpl}

\subsubsection{Generic elements used with partial gluing}

In \cite[Construction~7.19]{BRV1} we exhibit, for a Zariski closed
algebra $A$ over a field $F$, an explicit generic algebra,
generated by glued sums of sufficiently long products of generic
elements. This algebra then generates the variety of
$F$-identities of $A$. The easiest example is given below as
\eq{G1} in  \Eref{generic2}, where we adjust the construction to
more general base rings. When representing generic algebras, one
may refine the gluing procedure, as indicated in the following
sequence of examples.

\begin{exmpl}\label{generic2}
\begin{enumerate}
\item\label{G1}  The generic upper triangular $2\times 2$ matrix
$Y_i = \left(\AR{\xi_{i1} & \xi_{i3} \\ 0 & \xi_{i2}}\right)$ is
defined over the polynomial algebra $C= F[\xi_{ij} : i \in \N, 1
\le j \le 3 ];$ we get the {\bf generic algebra of upper
triangular matrices} by taking the subalgebra of $\M[2](C)$
generated by the~$Y_i $.

\item In \eq{G1}, we can glue the two diagonal components, by
replacing $C$ by the algebra
$$\bar C = F[\xi_{ij} : i \in \N, 1 \le j \le 3]/ \ideal{\xi_{i1}-\xi_{i2}
}.$$

\item\label{G4}  We can modify the gluing procedure to leave only
an ``infinitesimal part'' unglued. For example, let us replace $C$
by
$$\bar C = F[\xi_{ij}: i \in \N, 1 \le j \le 3]/ \langle
(\xi_{i_1,1}-\xi_{i_1,2}) (\xi_{i_2,1}-\xi_{i_2,2})\rangle.$$
Setting $\xi_{i2} = \xi_{1i}+\vareps_i$, one could think of $\bar
C$ as generated by $\xi_{i1}, \xi_{i3}$ and $\vareps_i$ ($i \in
\N$), where $\vareps_{i_1} \vareps_{i_2} = 0$ for all $i_1,i_2$,
and then the algebra of matrices is generated by $Y _i =
\smat{\xi_{i1}}{\xi_{i3}}{0}{\xi_{i1}+\vareps_i}$. Thus we view
the $\vareps_i$ as ``infinitesimal elements of second order,'' in
the sense that the product of any two of them is zero. From this
point of view, we are {\bf gluing} the two diagonal entries {\bf
up to infinitesimals}.

This example is symmetric, in the sense that the matrices $Y _i =
\smat{\xi_{i2}-\vareps_i}{0}{0}{\xi_{i2}}$, so one could reverse
the roles of $\xi_{i1}$ and $\xi_{i2}.$

\item More generally, for any $k>1,$ take $$\bar C =
F[\xi_{i1},\xi_{i2},\xi_{i3}: i \in I ]/ \langle
(\xi_{i1}-\xi_{i2}): i \in I \rangle^k.$$ One could think of the
algebra of \eq{G4} instead as generated by matrices $Y _i =
\smat{\xi_{i1}}{\xi_{i3}}{0}{\xi_{i1}+ \vareps_i}$ where $(\sum F
\vareps_i)^k = 0$. Thus, we view the $\vareps_i$ as
``infinitesimal elements of order $k$,'' in the sense that the
product of any $k$ of them is zero.

\item Consider the ideal $I$ of the commutative polynomial algebra
$F[\xi_{ij}^{(\ell)}]$, generated by the ${\xi_{ij}^{(\ell)}}^q
-\xi_{ij}^{(\ell)}$.  The algebra $F[\xi_{ij}^{(\ell)}]/I^k$ can
be viewed as obtained by deforming a finite field by
``infinitesimal elements of order $k$.''

\item\label{G7} We can also consider generic upper triangular
matrices with gluing of the diagonal entries up to infinitesimals,
by adjoining extra indeterminates to $\bar C$ and considering
matrices of the form $$\left(\AR{\xi_{i1} & \xi_{i3} \\ 0 &
\xi_{i1}+ \epsilon } \right);$$ represented over a field, these
could be viewed as $3 \times 3$
matrices of the form $$\left(\begin{matrix}\xi_{i1} & \xi_{i3} & \xi_{i3} \\
0 & \xi_{i1} & \xi_{i2}\\ 0 & 0 & \xi_{i1}\end{matrix} \right).$$

\item\label{G8} We can complicate \eq{G1}, taking instead the
algebra generated by matrices $Y _i = \left(\AR{\xi_{i1} & * & *
\\ 0 & \xi_{i2} & *
\\ 0 & 0 & \xi_{i2}+ \vareps_i}\right)$ where the $\vareps_i $ are
``infinitesimal elements of second order,'' and $*$ denotes
arbitrary entries.
\end{enumerate}
\end{exmpl}

The algebra of Example \ref{generic2}.\eq{G8} is represented in
$\M[3](F[\vareps_i])$. Since $\vareps_i$ is nilpotent of order 2,
$\vareps_i$ can be identified in $\M[2](F(\xi_i))$ (for $\xi_i$ a
commuting indeterminate over $F$) with $\left(\AR{0 & \xi
\\ 0 & 0}\right)$. Thus, Example \ref{generic2}.\eq{G8} has a faithful
representation of degree $6$ over a field. (In fact, one could get
a representation of degree $4$ using the argument of Example
\ref{generic2}.\eq{G7}.) When we increase the order of an
infinitesimal, we must increase the order of the faithful
representation (over a field) accordingly, since an infinitesimal
$\vareps$ of order $k$ is represented by the $k \times k$ matrix
$$\left(\AR{0 & \xi & 0 & \dots & 0 \\ 0 & 0 & \xi & \dots & 0 \\
\vdots & \vdots & \vdots & \ddots & \vdots \\ 0 & 0 & 0 & \dots &
\xi \\ 0 & 0 & 0 & \dots & 0}\right),$$ which is nilpotent of
order $k$. (Note that as the order of the representation
increases, so does the index of nilpotence of the radical.)

The generic description of partial gluing up to infinitesimals
becomes more complicated in the presence of a slight
generalization of Frobenius gluing, because we need to deal both
with the Frobenius automorphism and also with the degree of the
infinitesimal.

\begin{exmpl}\label{generic5} In order to introduce Frobenius gluing
of power $q$, up to infinitesimals, to Example
\ref{generic2}.\eq{G7}, we define generic triangular matrices $Y
_i = \left(\AR{\xi_{i1} & \xi_{i3}
\\ 0 & \xi_{i1}^q+ \vareps_i}\right)$ where each $\vareps_i
\vareps_j = 0;$ thus,
 we replace $C$ by
 the algebra
$$\bar C = F[\xi_{ij} : i \in I ,\ 1 \le j \le 3]/
\ideal{\xi_{i_1,1}^q-\xi_{i_1,3}}^2.$$ Stipulating the base field
$F$ to have order $q$
 requires extra relations of the form~
$(\xi_{ij}^{q}-\xi_{ij}).$

To describe the representation over a field, we must pass to
$$\bigg \{ \left(\begin{matrix} \a & * & * \\ 0 & \a ^q & * \\
0 & 0 & \a ^q \end{matrix}\right): \a \in F \bigg\}.$$
 The full quiver for this representation is the path
\vskip 0.2cm
\begin{equation}
\xymatrix@C=40pt{
I \ar@/^0pt/[r]
&
I^{(1)} \ar@/^0pt/[r]
&
I^{(1)}} .
\end{equation}
\vskip 0.6cm
\end{exmpl}

One  can also introduce Frobenius ``self-gluing'' of power $q$, up
to infinitesimals, by imposing the relation $(x^q-x)^\ell$ on a
vertex.

 All
examples of infinitesimal compression can be viewed in such a
manner, where the number of identifications at the end is the
order of the infinitesimal.

\subsection{Proportional gluing}

Off-diagonal gluing involves extra intricacies.  Here is an
example of a commutative algebra with non-identical gluing.

\begin{exmpl}\label{weight1}
Fix $\nu\in K$.
The algebra $\set{\left(\begin{matrix} \a & \beta & \gamma \\ 0 & \a & \nu \beta \\
0 & 0 & \a \end{matrix}\right) : \a, \beta, \gamma \in K}$ is a
\emph{commutative} algebra which is described by the full quiver
\vskip 0.2cm
\begin{equation}
\xymatrix@C=40pt{
I \ar@/^0pt/[r]^{\beta}
&
I \ar@/^0pt/[r] ^{\nu\beta}
&
I
}
\end{equation}
\vskip 0.6cm \noindent which denotes the relation $\lam_{21} =
\nu\lam_{32}$. Indeed, to verify commutativity, one may assume $\a
= 0,$ and checks that \begin{equation}
\left(\begin{matrix} 0& \beta & \gamma \\ 0 & 0 & \nu \beta \\
0 & 0 & 0 \end{matrix}\right)\left(\begin{matrix} 0 & \beta' & \gamma' \\ 0 & 0 & \nu \beta' \\
0 & 0 & 0 \end{matrix}\right) = \left(\begin{matrix} 0 & 0 & \nu \beta\beta' \\ 0 & 0 & 0 \\
0 & 0 & 0 \end{matrix}\right) = \left(\begin{matrix} 0 & \beta' & \gamma' \\ 0 & 0 & \nu \beta' \\
0 & 0 & 0 \end{matrix}\right)\left(\begin{matrix} 0 & \beta & \gamma \\ 0 & 0 & \nu \beta \\
0 & 0 & 0 \end{matrix}\right).
\end{equation}
\end{exmpl}

\begin{exmpl}\label{gluerad} We start with the paths
\begin{equation}\label{ex8.0}
\xymatrix@C=40pt{
\bullet \ar@/^0pt/[r]
&
\bullet \ar@/^0pt/[r]
&
\bullet \ar@/^0pt/[r]
&
\bullet }
\end{equation}
and
\begin{equation}\label{ex8.a}
\xymatrix@C=40pt{
I \ar@/^0pt/[r]
&
I \ar@/^0pt/[r]
&
I \ar@/^0pt/[r]
&
I }
\end{equation}

We recall that the full quiver of the diagram \eq{ex8.0}
corresponds to the algebra of upper triangular matrices; the full
quiver of the diagram \eq{ex8.a}, in which the blocks are glued,
corresponds to the algebra of unipotent upper triangular matrices
$$U_4 = \set{\(\begin{array}{cccc}
\alpha & * & * & *\\
0 & \alpha & * & *\\
0 & 0 & \alpha & *\\0& 0 & 0 & \alpha
\end{array}\) \suchthat \alpha \in F},$$
without radical gluing.
\vskip 0.0cm

Identical labels in the full quiver \eq{ex8.a} provide
identifications in the next layer of components, such as for the
algebra
$$U_4^+ = \set{\(\begin{array}{cccc}
\alpha & \beta & \gamma & \mu\\
0 & \alpha & \beta & \delta\\\
0 & 0 & \alpha & \beta \\
0& 0 & 0 & \alpha
\end{array}\) \suchthat \alpha, \beta, \gamma,\delta,\mu \in F},$$
whose full quiver is given in \eq{exJ2}.

In order to add the relation $\delta = \gamma $, we must bring the
non-primitive arrows back into the full quiver, which no longer
can be represented via  a single branch, as illustrated in diagram
\eq{exJ}.

But this can be compressed to a point. The corresponding algebra
$$B_4 = \set{\(\begin{array}{cccc}
\alpha & \beta & \gamma & \mu\\
0 & \alpha & \beta & \gamma\\
0 & 0 & \alpha & \beta \\
0& 0 & 0 & \alpha
\end{array}\) \suchthat \alpha, \beta, \gamma,\mu \in F}$$
is commutative, being isomorphic to the algebra
$F[\vareps]/\langle \vareps^4 \rangle$.

If instead we want the relation $\delta = \nu\gamma $, we get the
full quiver:

\begin{equation}\label{ex8.c1}
\xymatrix@C=40pt{
\comp[I]{1}{K} \ar@/^0pt/[r]|(0.5){\beta} \ar@/^14pt/[rr]|(0.5){\gamma}
&
\comp[I]{1}{K} \ar@/^14pt/[rr]|(0.25){\hole}|(0.51){\,\nu\gamma} \ar@/^0pt/[r]|(0.4){\beta }
&
\comp[I]{1}{K} \ar@/^0pt/[r]|(0.4){\beta}
&
\comp[I]{1}{K}},
\end{equation}

whose corresponding algebra is still commutative, since
$$\begin{aligned}\(\begin{array}{cccc}
0 & \beta & \gamma & \mu\\
0 & 0 & \beta & \nu\gamma\\
0 & 0 & 0 & \beta \\
0& 0 & 0 & 0
\end{array}\) & \(\begin{array}{cccc}
0 & \beta' & \gamma' & \nu\\
0 & 0 & \beta' & \nu\gamma'\\
0 & 0 & 0 & \beta' \\
0& 0 & 0 & 0
\end{array}\) = \(\begin{array}{cccc}
0 & 0 & \beta \beta' &\nu (\beta \gamma'+ \beta '\gamma )\\
0 & 0 & 0 & \beta \beta'\\
0 & 0 & 0 & 0\\
0& 0 & 0 & 0
\end{array}\) \\ & = \(\begin{array}{cccc}
0 & \beta' & \gamma' & \nu\\
0 & 0 & \beta' & \nu\gamma'\\
0 & 0 & 0 & \beta' \\
0& 0 & 0 & 0
\end{array}\)\(\begin{array}{cccc}
0 & \beta & \gamma & \mu\\
0 & 0 & \beta & \nu\gamma\\
0 & 0 & 0 & \beta \\
0& 0 & 0 & 0
\end{array}\).\end{aligned} $$
\end{exmpl}

\subsection{Gluing between branches -- Permuted gluing}\label{ss:54}
The theory becomes considerably more intricate when we need to
consider gluing between several different branches in the same
full quiver. The same kinds of gluing (identical, Frobenius,
proportional Frobenius) are involved.

\begin{defn}
Gluing between branches is called {\bf total}
if for each arrow in one branch there is an
arrow of the other branch, glued  proportionally to it.
\end{defn}
Total gluing need not involve a 1:1 correspondence of the arrows;
e.g.~the algebra  $$  \set{\(\begin{array}{ccccc}
\alpha & \beta & *&\gamma & *\\
0 & \alpha & \beta & 0 & *\\
0 & 0 & \alpha & 0 & \gamma \\
0& 0 & 0 & \alpha & \beta \\
0& 0 & 0 & 0 & \alpha
\end{array}\) \suchthat \alpha, \beta, \gamma \in F}$$
whose full quiver is
\begin{equation}\label{---}
\xymatrix@C=40pt{
\comp[I]{1}{K} \ar@/^0pt/[r]|(0.5){\beta} \ar@/^14pt/[rrr]|(0.5){\gamma}
 \ar@/_9pt/[rr] \ar@/_18pt/[rrr]
&
\comp[I]{1}{K} \ar@/^0pt/[r]|(0.4){\beta } \ar@/_18pt/|(0.125){\hole}|(0.333){\hole}[rrr]
&
\comp[I]{1}{K} \ar@/^14pt/[rr]|(0.20){\hole}|(0.51){\,\gamma}
&
\comp[I]{1}{K} \ar@/^0pt/[r]|(0.4){\beta}
&
\comp[I]{1}{K}
};
\end{equation}
as usual undecorated arrows stand for non-glued entries.
\begin{defn}\label{permglu}
Total gluing between two branches is {\bf permuted} if it involves a
permutation of the arrows. (In other words, there is a graph
isomorphism which preserves the labels of vertices, and which also
preserves the labels of arrows up to multiplying by constants).

{\bf Degenerate gluing} is the special case of permuted gluing in
which one branch is glued identically, arrow for arrow, with
another branch (i.e., via the identity permutation)\end{defn}
Degenerate gluing is illustrated in the following quiver, for
fixed $\lambda,\lambda' \in F$:
\begin{equation}\label{E.G2}
\xymatrix@C=40pt@R=32pt{
\comp[I]{1}{K} \ar@/^0pt/[r]^{\alpha} \ar@/_0pt/[d]^{\lambda \alpha}
&
\comp[\II]{1}{K} \ar@/^0pt/[d]^{\beta}
\\
\comp[\II]{1}{K} \ar@/^0pt/[r]^{\lambda'\beta}
&
\comp[\III]{1}{K}
}
\end{equation}

\begin{lem}\label{permut2}
Any full quiver is a union of full quivers with total gluing.
\end{lem}
\begin{proof}
We  say that two branches are equivalent if they have permuted
gluing between them. This is obviously an equivalence, so we just
partition the branches of the full quiver by their equivalence
classes.
\end{proof}

When dealing with degenerate gluing of paths, there are two separate cases to consider.
\begin{exmpl}
Consider the matrix representation of the full quiver
\eqref{E.G2}, where the components form the diagonal blocks of
$\M[4](K)$ in the natural manner. According to the given
proportional gluing, the radical is generated by $e_{12}+\lambda
e_{13}$ and $e_{24}+\lambda'e_{34}$, whose product is
$(1+\lambda\lambda')e_{14}$. If $1+\lambda\lambda' \neq 0$, then
the algebra is isomorphic to the one represented via  the full
quiver $I \ra \II \ra \III$. Otherwise, there is no map from $I$
to $\III$, and the algebra can be represented via  the
glue-connected full quiver $I\ra \II, \, \II \ra \III$.
\end{exmpl}

Using generic commuting elements to expand the base ring, one can
remove all degenerate gluing in a more systematic manner.

\begin{rem}\label{remivedeg}
One can transform degenerate gluing of paths in a full quiver into
a single path (at the expense of passing to algebras over
commutative affine algebras which are not necessarily fields).
Namely: if the products of the primitive arrows sum up to a
non-zero map, replace the glued paths by a single branch having
the same attributions. Otherwise, if the sum is zero, attach a
formal commuting indeterminate $\xi_i$ to each gluing class of
arrows, and define the appropriate relations on the commutative
polynomial algebra $F[\xi]$ to obtain a commutative affine algebra
for which the degeneration is described in terms of a single branch.

For example, for the full quiver given in \eqref{E.G2} we replace
$\a$ by $\xi_1\a$ and $\beta$ by $\xi_2\beta,$ and let $C =
F[\xi_1, \xi_2]/\langle \xi_1 \xi_2\rangle$.
Then we can represent our algebra over $C$, and it is intuitively clear  that the full quiver is the single path
\begin{equation}
\xymatrix@C=40pt{
\comp[I]{1}{K} \ar@/^0pt/[r]^{\xi_1\alpha}
&
\comp[\II]{1}{K} \ar@/^0pt/[r]^{\xi_2\beta}
&
\comp[\III]{1}{K}
},
\end{equation}
i.e. $\set{\(\begin{array}{ccc} C & \xi_1C & 0 \\ 0 & C & \xi_2C \\ 0 & 0 & C \end{array}\)}$.

Conversely, suppose our algebra is representable over a
commutative affine $F$-algebra~$C$. Since $C$ is representable we
can represent our algebra over $F$, but at the expense of a more
complicated full quiver.
\end{rem}

We turn to examples of permuted gluing:

\begin{exmpl}\label{Inf2}
The branches $$I \to I \to \II \to I, \qquad I \to \II \to I \to
I$$ are not glued, since the arrows are not labelled. On the other
hand, the gluing between the branches
\begin{equation}
\xymatrix@C=30pt{
I \ar@/^0pt/[r]^{\alpha}
&
I \ar@/^0pt/[r]^{\beta}
&
\II \ar@/^0pt/[r]^{\gamma}
&
I {}
},
\end{equation}
\begin{equation}
\xymatrix@C=30pt{
I \ar@/^0pt/[r]^{\beta}
&
\II \ar@/^0pt/[r]^{\gamma}
&
I \ar@/^0pt/[r]^{\alpha}
&
I {}
}
\end{equation}

is permuted.

Likewise, the gluing between the branches
\begin{equation}
\xymatrix@C=30pt{
I \ar@/^0pt/[r]^{\alpha}
&
\II \ar@/^0pt/[r]^{\beta}
&
I \ar@/^0pt/[r]^{\gamma}
&
\II {}
},
\end{equation}
\begin{equation}
\xymatrix@C=30pt{
I \ar@/^0pt/[r]^{\gamma}
&
\II \ar@/^0pt/[r]^{\beta}
&
I \ar@/^0pt/[r]^{\alpha}
&
\II {}
}
\end{equation}
is permuted.
\end{exmpl}

Note that the presence of permuted gluing dictates certain gluing
of vertices (namely, the vertices corresponding to the glued
arrows). We say that permuted gluing is {\bf complete} if each
arrow is glued to an initial arrow of (at least one) branch.
(Here, by initial arrow, we mean the first arrow in the branch.)

We can characterize the full quiver with complete gluing. In order
to avoid the trivial situation of a single arrow (so that the full
quiver has length 1), we restrict our attention to length at least
2.

\begin{prop}\label{completeglue}
Any full quiver (of length $>1$) with complete gluing of $m$
primitive arrows is a sub-quiver of a cube, of dimension $k \geq
2$, in which all the vertices are glued. For example, if there are
three primitive arrows, the full quiver has the form \vskip 0.2cm
\begin{equation}\label{ex8+11}
\xymatrix@C=24pt@R=24pt{
    {}
    & I \ar[rr]|(0.5){\alpha} \ar[dd]|(0.25){\,\gamma\,}|(0.501)\hole \ar[dl]|(0.5){\beta}
    & {}
    & I \ar[dd]|(0.5){\nu_{3}\gamma} \ar[dl]|(0.5){\nu_{2}\beta}
\\
    I \ar[rr]|(0.25){\nu_{1}\alpha} \ar[dd]|(0.5){\nu_{3}'\gamma}
    & {}
    & I \ar[dd]|(0.3){\nu_{3}''\gamma}
    & {}
\\
    {}
    & I \ar[rr]|(0.25){\nu_{1}'\alpha}|(0.5)\hole \ar[dl]|(0.45){\nu_{2}'\beta}
    & {}
    & I \ar[dl]|(0.5){\nu_{2}''\beta}
\\
    I \ar[rr]|(0.5){\nu_{1}''\alpha}
    & {}
    & I
    & {}
}
\end{equation}

\end{prop}
\begin{proof} First of all, the assumption of complete gluing
means that every arrow starts with a vertex labelled~I, which
means that all vertices except the last are glued to a vertex
labelled~I. But then the hypothesis of length $>1$ implies the
last vertex also is glued to the vertex ~I, since its arrow is
glued to an arrow $I \to I.$

It remains to observe that all possible branches can be described
inside the cube such that each arrow from the initial vertex is
labelled $\a, \beta, \gamma,$ etc., and all other arrows are
proportionally glued. (This can be seen inductively by projecting
onto a face of the cube.)
\end{proof}

\begin{exmpl}\label{Grass}
Perhaps the most well-known example of permuted gluing is given by
the Grassmann algebra $G = \sum _{i=1}^n Fe_i$ on $n$ generators,
defined by the relations $e_i^2 = 0$ and $e_i e_j = -e_j e_i$. $G$
has no nontrivial idempotents, so all vertices are glued. We start
with two generators: \vskip 0.0cm
\begin{equation}\label{G2+}
\xymatrix@C=40pt@R=32pt{
I \ar@/^0pt/[r]^{\alpha} \ar@/_0pt/[d]^{\beta}
&
I \ar@/^0pt/[d]^{-\beta}
\\
I \ar@/^0pt/[r]^{\alpha}
&
I
}
\end{equation}

Likewise, one has the full quiver of the Grassmann algebra without
1, on two generators: \vskip 0.0cm
\begin{equation}\label{G2.1}
\xymatrix@C=40pt@R=32pt{
\circ \ar@/^0pt/[r]^{\alpha} \ar@/_0pt/[d]^{\beta}
&
\circ \ar@/^0pt/[d]^{-\beta}
\\
\circ \ar@/^0pt/[r]^{\alpha}
&
\circ
}
\end{equation}

For the variety determined by the $3$-generated Grassmann algebra
(with unit element), the full quiver is the following cube:

\vskip 0.2cm
\begin{equation}\label{ex8+11.2}
\xymatrix@C=18pt@R=18pt{
{} & I \ar[rr]^{\alpha} \ar[dd]|(0.5)\hole_(0.7){\gamma} \ar[dl]_{\beta} & {} & I \ar[dd]^{-\gamma} \ar[dl]_(0.6){-\beta}
\\
I \ar[rr]^(0.7){\alpha} \ar[dd]_{-\gamma} & {} & I \ar[dd]_(0.7){\gamma} & {}
\\
{} & I \ar[rr]|(0.5)\hole^(0.7){\alpha} \ar[dl]_(0.4){\beta} & {} & I \ar[dl]^{-\beta}
\\
I \ar[rr]^{\alpha} & {} & I & {}
}
\end{equation}

Analogously, the full quiver for the $m$-generated Grassmann
algebra is the corresponding $m$-dimensional unit cube, with the
analogous permuted gluing. (Although the full quiver can be nicely
projected onto the plane for $m =3$, this becomes impossible for
larger $m$.)
\end{exmpl}

\medskip

Here is an interesting related example.
\begin{exmpl}
Let us consider the full quiver of the Grassmann envelope $E$ of
the superalgebra $A = {\mathbb M}_{n,k}$, which is defined as $E = G_0 \!
\otimes\! A_0 \ \oplus\ G_1\! \otimes\! A_1,$ where the Grassmann
algebra $G = G_0 \oplus G_1$, and $A_0$ is the sum of the diagonal
$n\times n$ and $k \times k$ blocks, whereas $A_1$ is the sum of
the off-diagonal $n\times k$ and $k \times n$ blocks. The full
quiver of the $m$-generated Grassmann envelope of the superalgebra
${\mathbb M}_{n,k}$ looks like two copies of the double of the
full quiver of $G$, and may be described as follows:

Ordered pairs of vertices denoting blocks of size $n$ and $k$
respectively are located at each corner of the $m$-dimensional
cube. We color the first component of each pair as red, and the
second component as blue. In other words, the red components
correspond to blocks of size $n$ and are glued among themselves,
whereas the blue components correspond to blocks of size $k$, and
are glued among themselves. These vertices correspond to the
idempotents of $G_0 \otimes A_0.$

The arrows correspond to elements of $ G_1 \otimes A_1,$ and occur
in pairs (in analogy to the full quiver of $G$.) For each pair,
one arrow connects the red component of the first vertex to the
blue component of the second vertex, and another connects the blue
component of the first vertex to the red component of the second
vertex. Thus, there are two components, each of which looks like
$G$. Each component corresponds to the same variety, but this
description corresponds to the canonical representation of the
superalgebra.

For example when there are $m = 2$ generators, the full quiver of $A = {\mathbb M}_{n,k}$ is
\begin{equation}\label{ex8+11.3}
\xymatrix@C=24pt@R=24pt{
\comp[I]{n}{K} \ar[r]^{\alpha} \ar[d]_{-\gamma} & \comp[\II]{k}{K} \ar[d]^{\gamma}  &
{} & \comp[\II]{k}{K} \ar[r]^{\beta} \ar[d]_{\delta} & \comp[I]{n}{K} \ar[d]_{-\delta}
\\
\comp[\II]{k}{K} \ar[r]^{\alpha}  & \comp[I]{n}{K} & {} & \comp[I]{n}{K} \ar[r]^{\beta}  & \comp[\II]{k}{K}
}
\end{equation}
\end{exmpl}

\section{Improving the full quiver}

In this section, we introduce techniques which enable us to choose
the representation for which the full quiver has the ``best''
form. This quest, which is difficult enough when the base field is
infinite, requires special techniques when the base field is
finite.

\subsection{Irreducible varieties}

By {\bf quasi-linear} variety, we mean an affine variety defined
by quasi-linear equations. (Thus, for $F$ infinite, quasi-linear
varieties are defined by linear equations.)

\begin{lem}\label{prereq1}
Any quasi-linear variety $V$ over a finite field $F$ is the direct
sum of the connected component $V_0$ of $0$, which is irreducible,
and a space $V_1$ which is finite dimensional over $F$ and thus
finite.
\end{lem}
\begin{proof}
$V/V_0$ is a finite set, and also a vector space over $F$, and
thus has some finite base $\bar b_1, \dots, \bar b_k$, which we
lift back to $b_1, \dots, b_k \in V$. $V_1 = \sum Fb_i$ is
annihilated by $p = \cha F$, and is a subvariety since $V$ is
quasi-linear, so we have split $V$ as $V_0\oplus V_1$.\end{proof}

We also need the following well-known fact from affine geometry:

\begin{rem}
Suppose $V$ is an irreducible affine variety. Then the coordinate
algebra $\mathcal O$ is the projective limit
$\underleftarrow{\lim} {\mathcal O}_s/I_s^n$ where $\mathcal O_s$
is the local ring of coordinate functions, and $I_s$ is the ideal
of functions vanishing at $s$, where $s$ is any fixed point of
$V$.
\end{rem}

In order to treat all irreducible varieties at once, we work in
the set of formal power series $F[[\la _1, \dots, \la _m]]$ over a
field $F$ with $|F|=q$.

\begin{defn}
A {\bf $q$-power series} is a power series $f(\la_1, \dots,
\la_m)$ in finitely many commuting
 indeterminates $\la_1, \dots, \la_m,$ with $f$ having constant term 0,
and all of whose finite truncations are $q$-polynomials. The
$q$-power series $f$ has {\it order 1} in $\la_i$ if some monomial
of $f$ has degree $1$ in $\la_i$. We let $H$ denote the subset of
$F[[\la _1, \dots, \la _m]]$ consisting of all $q$-power series.
\end{defn}

By definition
$$H = \left\{ \sum _{i=1}^m \sum _{j=0}^\infty
\a_{ij} \la_i^{q^{j}} \suchthat \a_{ij} \in F\right\},$$ but, for
example, $\la_1\la_2 \notin H$. Thus, $f \in H$ when each monomial
of $f$ is a $q^t$-power of some $\la_i$ (which may vary according
to the choice of monomial); $f$ has order $1$ in $\la_i$ when some
$\a_{i0}\ne 0$.

\begin{rem}\label{reducts}
\begin{enumerate}
\item\label{red1} If $f(\la_1, \dots, \la_m), g(\la_1, \dots, \la_m) \in H$,
with $g$ of order $i$ in $\la_i$, then $$f(\la_1, \dots,\la
_{i-1},g(\la_1, \dots, \la_m), \la _{i+1}, \dots, \la_m)\in H.$$
(In other words, $H$ is closed under composition of power series,
as are linear mappings.) Indeed, writing $g(\la_1, \dots, \la_m) =
\sum _{i=1}^m \sum _{j=0}^\infty \beta_{ij} \la_i^{q^{j}},$ we see
that for any $q$-power $q'$,
$$g(\la_1, \dots, \la_m)^{q'} = \sum _{i=1}^m \sum _{j=0}^\infty
\beta_{ij} \la_i^{q^{j}q'},$$ since $|F| = q$ and $\cha F = p$.

\item\label{red2} If $g = \sum _{i=1}^m \sum _{j=0}^\infty \a_{ij}
\la_i^{q^{j}}\in H$ and we take $ \a_{iu}\ne 0$ minimal among all
the nonzero coefficients of $g$, then $$g^{q^{-u}} = \sum _{i=1}^m
\sum _{j=0}^\infty \a_{ij} \la_i^{q^{j-u}},$$ which has order $1$ in
$\la_i$. Thus, we can always apply \eq{red1} with $g^{q^{-u}}$.
\end{enumerate}
\end{rem}

The following result is the key to studying relations arising from
$q$-polynomials. We say that variables are \textbf{independent} if
there are no weak Frobenius relation among them.

\begin{prop}\label{proppow}
Suppose $V_0$ is an irreducible quasi-linear affine variety
(containing $0$). Then there exist independent variables
$\la_{i_1}, \dots, \la_{i_s}$, such that all other $\la_i$ can be
expressed in terms of $q$-power series in $\la_{i_1}, \dots,
\la_{i_s}$.

(This generalizes the following fact: In any system of linear
equations, one can renumber the indices such that $\la_1, \dots,
\la_s$ are independent, and the other $\la _i$ are linearly
spanned by $\la_1, \dots, \la_s$.)
\end{prop}
\begin{proof}
When $\cha F = 0$, then ``quasi-linear'' means ``linear'' and
there is nothing to prove. Thus, we may assume that $\cha F = p
>0,$ and need to show that whenever we have a power series
equation involving indeterminates $\la_i$ expressed as an element
of $H$, that we can solve some indeterminate in terms of the
others. In view of Remark~ \ref{reducts}.\eq{red2}, we may assume
that this has order $1$ in some $\la_i$. Thus we have
$$ \sum _{i=1}^m \sum _{j=0}^\infty \a_{ij} \la_j^{q^{j}} = 0;$$
with $\a_{i,0}\ne 0;$ dividing through by $\a_{i,0},$ we may
assume that $\a_{i,0} = 1$. To ease notation, we assume that $i =
1$. Now we have $\la_1= \sum _{j=1}^\infty \a_{j} \la_1^{q^{j}}
+\sum _{i=2}^m \sum _{j=0}^\infty \a_{ij} \la_j^{q^{j}}$. Thus
substituting the right side for $\la_1$ yields another $q$-power
series in which we have increased the degree of $\la_1$ (and
furthermore have not produced new monomials of lower total
degree). Each stage produces a new power series in which the
degree of $\la_1$ in the monomials keeps increasing, so continuing
indefinitely yields an element of $\mathcal{O} = F(V_0)$ in which
$\la_1$ no longer appears, as desired.
\end{proof}

Note that the final argument does not apply in $H$ itself, as
$\lam_1$ cannot be eliminated from the  relation $\lam_1 =
\lam_1^q$; indeed, this relation defines a reducible variety.

\subsection{Reduction to proportional Frobenius gluing}

In general, any type of gluing can occur in a representation.
Indeed let $S$ be any set of $q$-polynomials on the $m^2$
variables $\Lambda = (\lam_{ij})_{i, j = 1,\dots,m}$. Then $B_S =
\set{\smat{\alpha I}{\Lambda}{0}{\alpha I} \suchthat \alpha \in F,\,
\forall f\in S\!: f(\Lambda) = 0}$ is a Zariski closed algebra whose
radical satisfies the relations in $S$.

We aim for the result that one can always find a representation of
a \Zcd\ algebra (over a suitable field extension) whose
off-diagonal gluing is a consequence of proportional Frobenius
gluing. This only seems to be accessible when the base field $F$
is infinite, so we start with this case; afterwards, we obtain the
analogous result for arbitrary $F$ in case the algebra is
relatively free.

\begin{rem}
As explained in \Ssref{ss:Zc}, the relations on any semisimple
subalgebra are of $F$-Frobenius type.
\end{rem}

\subsubsection{The case when the base field $F$ is infinite}

When $F$ is infinite,  the relations are linear. Then we may
partition the arrows to independent variables, in terms of which
all the other variables can be expressed. Suppose $\gamma$ is a
dependent arrow, satisfying a relation of the form $\gamma =
\sum_{i = 1}^{k}\theta_i \alpha_i$, where $\alpha_i$ are
independent, and $\theta_i \in F$.

Replace the arrow $\gamma$ by $k$ arrows, identically gluing the
$k$ initial vertices and (separately) the $k$ terminal vertices,
as well as the resulting $k$ copies of any incoming or outgoing
arrow from these vertices. The new arrows become
$\gamma_1,\dots,\gamma_k$, with the relations $\gamma_i = \theta_i
\alpha_i$, so we only have proportional Frobenius gluing.

If the arrow $\gamma$ in the original full quiver shares no vertex
with other arrows representing the $\gamma_i$, then the new full
quiver represents the same algebra as the original one. However in
general the procedure does not work, as the following example
illustrates.

\begin{exmpl}\label{prop-zation1}
The path \begin{equation}\label{38}
\xymatrix@C=30pt{
\comp[I]{1}{K} \ar@/^0pt/[r]|(0.5){\alpha}
&
\comp[I]{1}{K} \ar@/^0pt/[r]|(0.5){\beta}
&
\comp[I]{1}{K} \ar@/^0pt/[r]|(0.5){\gamma}
&
\comp[I]{1}{K}
}
\end{equation}
with the linear relation $ \alpha = \beta+ \gamma$ could be
expanded to
\begin{equation}
\xymatrix@C=30pt@R=2pt{
\comp[I]{1}{K} \ar@/^0pt/[r]|(0.5){\beta}
&
\comp[I]{1}{K} \ar@/^0pt/[r]|(0.5){\beta}
&
\comp[I]{1}{K} \ar@/^0pt/[r]|(0.5){\gamma}
&
\comp[I]{1}{K}
\\
\\
\comp[I]{1}{K} \ar@/^0pt/[r]|(0.5){\beta}
&
\comp[I]{1}{K} \ar@/^0pt/[r]|(0.5){\beta}
&
\comp[I]{1}{K} \ar@/^0pt/[r]|(0.5){\gamma}
&
\comp[I]{1}{K} }
\end{equation}
or to
\begin{equation}
\xymatrix@C=30pt@R=2pt{
{}
&
\comp[I]{1}{K} \ar@/^0pt/[r]|(0.5){\beta}
&
\comp[I]{1}{K} \ar@/^0pt/[r]|(0.5){\gamma}
&
\comp[I]{1}{K}
\\
\comp[I]{1}{K} \ar@/^0pt/[ru]|(0.5){\beta} \ar[rd]|(0.5){\gamma} & {} & {} & {}
\\
{}
&
\comp[I]{1}{K} \ar@/^0pt/[r]|(0.5){\beta}
&
\comp[I]{1}{K} \ar@/^0pt/[r]|(0.5){\gamma}
&
\comp[I]{1}{K} }
\end{equation}

However in the original algebra, the generators $b$ and $c$
corresponding to the arrows marked $\beta$ and $\gamma$,
respectively, satisfy the relation $b^2 = cb$; this relation does
not hold in the expanded algebras. The situation can be remedied
by further identification; the full quiver \eq{FQ38} does
represent the same algebra as the one in \eq{38}.

\begin{equation}\label{FQ38}
\xymatrix@C=30pt@R=2pt{
{}
&
\comp[I]{1}{K} \ar@/^0pt/[rd]|(0.5){\beta}
&
{}
&
{}
\\
\comp[I]{1}{K} \ar@/^0pt/[ru]|(0.5){\beta} \ar[rd]|(0.5){\gamma} & {} & \comp[I]{1}{K} \ar[r]|(0.5){\gamma} & \comp[I]{1}{K}
\\
{}
&
\comp[I]{1}{K} \ar@/^0pt/[ru]|(0.5){\beta}
&
{}
&
{} }
\end{equation}

However, consider the same quiver, with $\alpha$ and $\gamma$
taken to be the independent variables, and $\beta = \alpha -
\gamma$. The generators $a$ and $c$ corresponding to $\alpha$ and
$\gamma$ satisfy $a^2+ac+c^2 = 0$. This relation does not hold for
any intermediate expansion between \eq{38} and
\begin{equation}
\xymatrix@C=30pt@R=2pt{
\comp[I]{1}{K} \ar@/^0pt/[r]|(0.5){\alpha}
&
\comp[I]{1}{K} \ar@/^0pt/[r]|(0.5){\alpha}
&
\comp[I]{1}{K} \ar@/^0pt/[r]|(0.5){\gamma}
&
\comp[I]{1}{K}
\\
{}
\\
\comp[I]{1}{K} \ar@/^0pt/[r]|(0.5){\alpha}
&
\comp[I]{1}{K} \ar@/^0pt/[r]|(0.5){-\gamma}
&
\comp[I]{1}{K} \ar@/^0pt/[r]|(0.5){\gamma}
&
\comp[I]{1}{K} }
\end{equation}

\end{exmpl}

A general algorithm would be rather intricate to describe,
and is best done generically, motivating the proof of our next result.

\begin{thm}\label{propglu}
The \Zcr\ of any representable affine PI-algebra $A$ over an
infinite field $F$ has a representation and full quiver all
of whose polynomial relations  are consequences of
proportional gluing.
\end{thm}
\begin{proof}
The radical $J$ of the Zariski closure $\hat A$ satisfies $J^\ell
= 0$ for some $\ell$, by Lemma~\ref{nilrad}. The proof is
comprised of choosing basic linear functionals $\phi_i$
corresponding to the arrows of the original quiver. We write down
the dependences of these arrows, and define a representation over
a larger field in which we can obtain proportional relations
involving the coefficients of these dependences, thereby producing
a new quiver in which all of these dependences are described by
proportional Frobenius gluing, which apply to the original linear
functionals.

Since the base field $F$ is infinite, the Zariski closure $\hat A$
of the algebra $A$ is finite dimensional over $K$. Furthermore,
all gluing among vertices is identical. We view $\hat A \subseteq
\End _K V$ where $\dimcol{V}{K} = n$. Now we turn to the
sub-Peirce decomposition of Remark~ \ref{Peirce0}, which is
induced by the sums of idempotents along identically glued blocks.
Some of the generators of $A$ belong to the semisimple component,
and some to the radical component. We consider the finitely many
sub-Peirce components of these radical elements, and their
finitely many nonzero products (since all nonzero products have
length $\le \ell$).

Applying Gauss elimination to a basis of the space of linear
dependences (as in \cite[Corollary 6.5]{BRV1}), we may choose
indices $i_1 < i_2 < \dots < i_m$ such that there is no relation
involving only the $\alpha_{i_j}$, while every other variable
$\alpha_u$ can be expressed in their terms, via relations of the
form
\begin{equation}\label{lindep1}
\alpha_{u} = \sum _{j=1}^{m} \theta_{j,u}
\alpha_{i_j}\end{equation} for $\theta_{j,u} \in F$.

Note that all such linear relations would be repeated in any glued
vertices. Let $k(\alpha_{u})$ denote the number of summands in
\eqref{lindep1}.
We call each such $\alpha_{u}$ a {\it dispensable} arrow; the
arrows $\alpha_{i_1},\dots,\alpha_{i_m}$ are called {\it
indispensable}. Note that we really are talking about linear
functionals between the respective semisimple blocks, with their
gluing, since any linear dependence in \eqref{lindep1} must be
repeated for every arrow glued to $\alpha_{i_u}$. Thus, strictly
speaking, we are handling dispensable and indispensable linear
functionals.

To form our extended full quiver, we start by taking any
indispensable arrow $\a$ in the original full quiver, say from
vertex $v$ to $w$, and replacing each with $k$ vertices $v^{(1)},
\dots, v^{(k)},$ $w^{(1)}, \dots, w^{(k)},$ and arrows from
$v^{(i)}$ to $v^{(i)}$, all glued identically for each $1 \le i
\le k.$

Thus, for the  dispensable $\a _{i_u}$ in the original full
quiver, say from $v _{i_u}$ to $w _{i_u}$, we put in $k(\a_{i_u})$
classes of identically glued arrows, for each $i$ in the right
side of \eqref{lindep1}. We observe that under the corresponding
representation $A \mapsto \M[n'](C)$ the gluing has become
proportional. (In fact, most new gluing is identical; the only
proportional gluing is between the $\a _{i}$ from the
indispensable arrows and the $\bar \theta_i \a _{i}$ from the
dispensable arrows, as illustrated in \Eref{proportionalization}).
\end{proof}

The following example indicates just how far we can improve the
full quiver.

\begin{exmpl}\label{ALEXEI}
The proof of \Tref{propglu} separates the arrows of the representation
\begin{equation}\label{AL.1}
\xymatrix@C=30pt{
\circ \ar@/^0pt/[r]^{\alpha+\beta}
&
\circ \ar@/^0pt/[r]^{\alpha}
&
\circ \ar@/^0pt/[r]^{\beta}
&
\circ {}
}
\end{equation}
which are involved in non-proportional relations, to the form
\begin{equation}\label{AL.1.imp}
\xymatrix@C=30pt{
\circ \ar@/^6pt/[r]^{\alpha} \ar@/_6pt/[r]_{\beta}
&
\circ \ar@/^0pt/[r]^{\alpha}
&
\circ \ar@/^0pt/[r]^{\beta}
&
\circ {}
};
\end{equation}
Note the presence of multiple edges to the left.
\end{exmpl}

If one is willing to pass to relatively free algebras, one can
further improve the full quiver. Since Theorem~\ref{propglu}
permits us to reduce to the case that all gluing is proportional,
we weaken \Dref{permglu} slightly, and define an equivalence
on arrows when they are proportionally glued; two branches are
called \textbf{proportionally permuted} if for each arrow, the
number of arrows proportionally glued to it is the same in each of
the branches. (In other words, one branch is obtained from the
other by permuting the arrows, but perhaps changing the
proportionality constants in the gluing.)

\begin{thm}\label{propglurel}
When $F$ is infinite,  any relatively free affine PI-algebra $A$
has a representation whose polynomial relations are consequences
of proportional gluing, and such that there are no double edges
(between two adjacent vertices) in the full quiver. More
generally, any two branches between two vertices are
proportionally permuted.
\end{thm}
\begin{proof}
The radical $J$ of the Zariski closure $\hat A$ satisfies $J^\ell
= 0$ for some $\ell$, by Lemma~\ref{nilrad}. The proof is
comprised of choosing basic linear functionals $\phi_i$
corresponding to the arrows of the original quiver. We write down
the dependences of these arrows, and define a representation over
a larger field in which we can obtain proportional relations
involving the coefficients of these dependences, thereby producing
a new quiver in which all of these dependences are described by
proportional Frobenius gluing, which apply to the original linear
functionals.

We start as in the previous proof.
Thus, applying Gauss elimination to a basis of the space of linear
dependences (as in \cite[Corollary 6.5]{BRV1}), we may choose
indices $i_1 < i_2 < \dots < i_m$ such that there is no relation
involving only the $\alpha_{i_j}$, while every other variable
$\alpha_u$ can be expressed in their terms, via relations of the
form
\begin{equation}\label{lindep1rel}
\alpha_{u} = \sum _{j=1}^{m} \theta_{j,u}
\alpha_{i_j}\end{equation} for $\theta_{j,u} \in F$.

Note that all such linear relations would be repeated in any glued
vertices. Let $k(\alpha_{u})$ denote the number of summands in
\eqref{lindep1rel}.
We call each such $\alpha_{u}$ a {\it dispensable} arrow; the
arrows $\alpha_{i_1},\dots,\alpha_{i_m}$ are called {\it
indispensable}. Note that we really are talking about linear
functionals between the respective semisimple blocks, with their
gluing, since any linear dependence in \eqref{lindep1rel} must be
repeated for every arrow glued to $\alpha_{i_u}$. Thus, strictly
speaking, we are handling dispensable and indispensable linear
functionals.

Note that since the base field is infinite, $\hat A\otimes _F C$
is PI-equivalent to $\hat A.$ Thus, we may work in $\hat A\otimes
_F C$. We pass to the representation of $\hat A$ as
transformations of $V' = V \otimes _F C,$ where the
$\widehat{\theta_{i,u}}$ are commuting indeterminates over $F$ and
$C = F[\widehat{\theta_{i,u}}]/\langle \widehat{\theta_{i,u}}
\rangle^\ell.$ (Since we have only finitely many $\theta_{i,u}$
needed for our finitely many relations, we see that
$\dimcol{C}{F}<\infty,$ and thus $\dimcol{V'}{K} < \infty.$) Let
$\bar \theta _{i,u}$ denote the canonical image of
$\widehat{\theta _{i,u}}$ in $C$. Intuitively, the $\bar \theta
_{i,u}$ behave like commuting indeterminates, except that the
 product of $\ell$ of  these $\bar \theta _{i,u}$ is 0. Thus the
nonzero products $\tilde \theta$ of the $\bar \theta _{i,u}$
comprise a base $\Theta$ for $C$ over $F$; if $B $ is a base for
$V$, then $B \otimes \Theta$ is a base for $V'$ over $F$, of order
$n' = n[C\! :\! F]$.

To form our extended full quiver, we start by taking any
indispensable arrow $\a$ in the original full quiver, say from
vertex $v$ to $w$, and replacing it in the new quiver by arrows
from $v \otimes \tilde{\theta}$ to $w \otimes \tilde \theta$, all
glued identically for each $\tilde \theta$ in $\Theta.$

For the  dispensable $\a _{i_u}$ in the original full quiver, say
from $v _{i_u}$ to $w _{i_u}$, we put in $k(\a_{i_u})$ classes of
identically glued arrows; namely, for each $i$ in the right side
of \eqref{lindep1rel}, we put in the identically glued arrows
$\bar \theta_i \a _{i}$ from $v _{i_u}\otimes \tilde \theta $ to
$w _{i_u}\otimes \bar \theta _{i}\tilde \theta $. We observe that
under the corresponding representation $A \mapsto \M[n'](C)$ the
gluing has become proportional and the endpoints have been
separated unless the branches are proportionally permuted. (In
fact, most new gluing is identical; the only proportional gluing
is between the $\a _{i}$ from the indispensable arrows and the
$\bar \theta_i \a _{i}$ from the dispensable arrows, as
illustrated in \Eref{proportionalization}).
\end{proof}

Theorem~\ref{propglurel} will be improved in \cite{BRV4}, by the
introduction of a natural grading on the algebra.

\begin{exmpl}\label{proportionalization}
Consider again the full quiver \eq{38} of \Eref{prop-zation1}. The
nilpotence index is $\ell = 4$, and in the single relation $\alpha
= \beta+\gamma$; here $\beta$ and $\gamma$ are indispensible, and
$\alpha$ is dispensible. We thus take $C =
F[\theta_1,\theta_2]/\ideal{\theta_1,\theta_2}^4$, with the
natural basis $\set{\theta_1^r \theta_2^s\suchthat r+s \leq 3}$,
of cardinality $10$. The expanded full quiver is
\begin{equation}
\xymatrix@C=48pt@R=0pt{
\comp[I]{1}{K}
&
\comp[\II]{1}{K} \ar[r]|(0.5){\beta}
&
\comp[\III]{1}{K} \ar[r]|(0.5){\gamma}
&
\comp[\IV]{1}{K}
\\
\comp[I]{1}{K} \ar[ru]|(0.4){\theta_1 \beta}
\ar[rd]|(0.4){\theta_2\gamma}
&
\comp[\II]{1}{K} \ar[r]|(0.5){\beta}
&
\comp[\III]{1}{K} \ar[r]
|(0.5){\gamma}
&
\comp[\IV]{1}{K}
\\
\comp[I]{1}{K}
&
\comp[\II]{1}{K} \ar[r]|(0.5){\beta}
&
\comp[\III]{1}{K} \ar[r]|(0.5){\gamma}
&
\comp[\IV]{1}{K}
\\
\comp[I]{1}{K} \ar[ruu]|(0.4){\theta_1 \beta}
\ar[rd]
|(0.5){\theta_2\gamma}
&
\comp[\II]{1}{K} \ar[r]|(0.5){\beta}
&
\comp[\III]{1}{K} \ar[r]|(0.5){\gamma}
&
\comp[\IV]{1}{K}
\\
\comp[I]{1}{K} \ar[ruu]|(0.6){\theta_1 \beta}
\ar[rdd]|(0.5){\theta_2\gamma}
&
\comp[\II]{1}{K} \ar[r]|(0.5){\beta}
&
\comp[\III]{1}{K} \ar[r]|(0.5){\gamma}
&
\comp[\IV]{1}{K}
\\
\comp[I]{1}{K} \ar[ruu]|(0.4){\theta_1 \beta}
\ar[rdd]|(0.3){\theta_2\gamma}
&
\comp[\II]{1}{K} \ar[r]|(0.5){\beta}
&
\comp[\III]{1}{K} \ar[r]|(0.5){\gamma}
&
\comp[\IV]{1}{K}
\\
\comp[I]{1}{K}
&
\comp[\II]{1}{K} \ar[r]|(0.5){\beta}
&
\comp[\III]{1}{K} \ar[r]|(0.5){\gamma}
&
\comp[\IV]{1}{K}
\\
\comp[I]{1}{K} \ar[ruu]|(0.3){\theta_1 \beta}
\ar[rdd]|(0.5){\theta_2\gamma}
&
\comp[\II]{1}{K} \ar[r]|(0.5){\beta}
&
\comp[\III]{1}{K} \ar[r]|(0.5){\gamma}
&
\comp[\IV]{1}{K}
\\
\comp[I]{1}{K} \ar[ruu]|(0.5){\theta_1 \beta}
\ar[rd]|(0.5){\theta_2\gamma}
&
\comp[\II]{1}{K} \ar[r]|(0.5){\beta}
&
\comp[\III]{1}{K} \ar[r]|(0.5){\gamma}
&
\comp[\IV]{1}{K}
\\
\comp[I]{1}{K}
&
\comp[\II]{1}{K} \ar[r]|(0.5){\beta}
&
\comp[\III]{1}{K} \ar[r]|(0.5){\gamma}
&
\comp[\IV]{1}{K}
},
\end{equation}
(The rows correspond to $\theta_1^3,
\theta_1^2, \theta_1^2\theta_2, \theta_1, \theta_1 \theta_2, 1,
\theta_1\theta_2^2, \theta_2,\theta_2^2$ and $\theta_2^3$, in this
order, for graphical reasons.)
\end{exmpl}

The reduction to proportional Frobenius gluing has implications in
the PI-theory. In \cite[Example 6.1]{BRV3} we give an example of
an algebra with parametric identities, which are defined over
a commutative algebra extending of the base field $F$, but are
not defined over~$F$. However when we extend $F$ so as to obtain a
representation with proportional gluing, we see that the
parametric identities are now defined over the base ring.

One can see easily that identical gluing between branches yields
parametric sets of identities. However, if the gluing is identity
permuted, then the corresponding algebra does not have a
parametric set of identities, since  the proportionality
coefficients are~1. Details are to be given in \cite{BRV3}.

\subsubsection{The case of a finite base field}

When the base field $F$ is finite, the situation is more
intricate, since we have to deal with polynomial relations arising
from Frobenius gluing. This forces us to deal with relatively free
algebras.

\begin{thm}\label{propglu11}
Any relatively free affine PI-algebra $A$ has a representation for
whose full quiver all gluing is Frobenius proportional.
\end{thm}
\begin{proof} We are done by Theorem~\ref{propglurel} unless the base field $F$ is
finite. Let $p = \cha (F)$. By \cite[Theorem~7.10]{BRV1}, the
Zariski closure of $A$ has finite PI-rank, which means that it is
PI-equivalent to a representable affine PI-algebra. Taking its
generic elements, as constructed in \cite[Construction 7.14]{BRV1}
and studied in \cite[Theorem 7.15]{BRV1}, consider the algebra
$\tilde A$ generated by these generic elements and pass to its
Zariski closure $\hat A$.

In view of Lemma~\ref{prereq1}, we can decompose $J = V_0 \oplus
V_1$ as a direct sum of the connected component ~$V_0$ of 0 and a
space $V_1$ which is finite dimensional over $F$ and thus finite,
say with base $b_1, \dots, b_m$. We take $\tilde A$ generated by
$V_0$ and all elements $w_k = \sum _j \psi_{j,k} b_j$ of $V_1$, $1
\le k \le m$, where $\psi_{j,k}\in F$. Note that $\psi_{j,k}^q
=\psi_{j,k}$ since $|F| = q$. Then any element of $\tilde{A}$ can
be represented as a linear combination of products of length $ \le
\ell m$ of these elements. We need to replace the $w_k$ by generic
elements $\sum _j \overline{\psi_{j,k}} b_j$ of $V_1$, where
$\overline{\psi_{j,k}}$ are generic coefficients satisfying
$\overline{\psi_{j,k}}^q = \overline{\psi_{j,k}}$. These elements
do not belong to $F$, but rather to the finite dimensional
extension $F[\widehat{\psi_{j,k}}]/\langle \widehat{\psi_{j,k}}^q
-\widehat{\psi_{j,k}}\rangle$, where $\widehat{\psi_{j,k}}$ are
commuting indeterminates over $F$.

We project onto the variety $\bar V$ defined by $q$-polynomials of
degree $\le q^{\ell m}.$ Since the projection of a quasi-linear
variety is quasi-linear, we may work in $\bar V$, in which all our
$q$-power series project to $q$-polynomials. In view of
Proposition \ref{proppow}, there is some subset of the arrows
which is independent, and all other arrows can be written as
$q$-polynomials in terms of these arrows.

Again we pass to the representation of $\hat A$ over
$$(J\otimes _F F[\theta_{i,j}]/\langle \theta_{i,j} \rangle^\ell)
\,\oplus \, (F[\widehat{\psi_{j,k}}]/\langle
\widehat{\psi_{j,k}}^q -\widehat{\psi_{j,k}}\rangle);  $$ we can
eliminate arrows as in the case for $F$ infinite (by the same
technique of expanding the graph and labelling one arrow of each
relation as ``dispensable''), but this time our relations are in
terms of $q$-polynomials rather than linear combinations. Thus, in
this case we only can reduce to proportional Frobenius gluing.
\end{proof}

Note that Theorem~\ref{propglu11} applies to $A$ as an algebra
over the base field $F$, but not as a $K$-algebra. One can often
extend Theorem~\ref{propglu11} to reduce to the case that there
are no double edges (between two adjacent vertices); this issue is
to be treated in \cite{BRV3}.

This theory of full quivers
(especially \Tref{propglu11}) enables us to characterize Noetherian properties
of relatively free algebras, to be discussed in
\cite{BRV3}.

\section{Decomposing quivers}

Let $A \sub \M[n](K)$ be an algebra in Wedderburn block form, and
let $\Gamma$ be its full quiver. We turn now to the question of
whether the decomposition of the full quiver $\Gamma$ has a
matching decomposition of the algebra $A$ as a  subdirect product.
In \cite{BRV1} we discuss fine decompositions of $A$ related to
various sets of idempotents in $A$.  To every set of vertices
$\Gamma_0 \sub \Gamma$ there is an associated idempotent $e =
e(\Gamma_0)$, which is the sum of the idempotents  in $\M[n](K)$
corresponding to the matrix blocks of vertices in $\Gamma_0$.
However in general $e$ is  not an element of $A$ because of gluing
(e.g., $e_{11}$ is not an element of $B_{\ell}$ in
\Eref{compress}); nevertheless, we consider such idempotents,
ignoring gluing altogether.

\subsection{Decomposing quivers}

We say that a quiver $\Gamma_0$ is {\bf convex} if for every $i, k \in
\Gamma_0$, if $i \ra j_1 \ra \cdots \ra j_t \ra k$ is a path
connecting $i$ and $k$ in $\Gamma_0$, then $j_1,\dots,j_t \in
\Gamma_0$.
\begin{prop}
If $\Gamma_0$ is convex, with $e = e(\Gamma_0)$, then $a \mapsto
eae$ is a projection of algebras $A \ra eAe$ (without $1$).
(Indeed, writing $a = \sum e_{ij} a_{ij}$ and $b = \sum e_{ij}
b_{ij}$, convexity implies that $eaebe = eabe$).
\end{prop}

\begin{cor}\label{sdcover}
If every vertex of $\Gamma$ belongs to one of the convex
sub-quivers $\Gamma_1,\dots,\Gamma_r$ (which are not necessarily disjoint),
then $A$ is a subdirect product of the algebras $A_i$ corresponding to the $\Gamma_i$.
\end{cor}
\begin{proof}
The map $A \ra A_1 \times \cdots \times A_r$ is an algebra
homomorphism by convexity, and it is injective by the covering
assumption.
\end{proof}

\begin{exmpl}\label{pseud}
Suppose the quiver contains three vertices with arrows glued in
the form
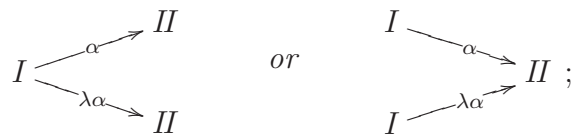
\begin{equation*}
\xymatrix@C=40pt@R=2pt{
{} & \comp[\II]{1}{K}
\\
\comp[I]{1}{K} \ar[ru]|(0.5){\alpha}
\ar[rd]|(0.5){\lambda\alpha} & {}
\\
{} & \comp[\II]{1}{K} }
\qquad
\begin{array}{c} {\ } \\ {\ } \\ \mbox{or} \end{array}
\qquad
\xymatrix@C=40pt@R=2pt{
\comp[I]{1}{K} \ar[rd]|(0.5){\alpha} &  {}
\\
{} & \comp[\II]{1}{K} \ ;
\\
\comp[I]{1}{K} \ar[ru]|(0.5){\lambda\alpha}  & }
\end{equation*}
for some $\lambda \in K$. (The vertices must be glued, in view of
\Pref{p2p}). Then we can trade two arrows for one, by creating a
new vertex, replacing the existing arrows with
$\xymatrix@C=32pt{\comp[I]{1}{K} \ar[r]^{(\lambda+1)\alpha} &
\comp[\II]{1}{K}}$. \end{exmpl}

To break down the full quiver further, we need another definition.

\begin{defn}\label{candef1}
A pair of arrows is {\bf separated} if  the glued components of
the two arrows are disjoint; i.e., there is no single branch that
contains arrows glued to each of them.

A {\bf   geometric decomposition} $\Gamma \ra \cup \Gamma_i$ of
basic full quivers  is a union of subquivers such that, for any
arrow of $\Gamma_i$, all arrows of $\Gamma$ glued to this arrow
also appear in $\Gamma_i$.

A full quiver is  {\bf geometrically decomposable} if it admits a
proper geometric decomposition, and geometrically \indecomposible\
otherwise.
\end{defn}

\begin{rem}\label{sep1} Any full quiver with a pair of separated arrows is geometrically decomposable.
\end{rem}

We would like the geometric decomposition of quivers to match a
``geometric'' subdirect decomposition of algebras. However, we are
confronted with the following sort of example:

\begin{exmpl}\label{mess}
Consider the algebra $A$ having full quiver $\Gamma$ given by:
\begin{equation} \label{EE211}
\xymatrix@C=30pt@R=2pt{
{} & \comp[I]{1}{K} \ar[rd]|(0.5){\beta} & {}
\\ \comp[I]{1}{K} \ar[ru]|(0.5){\alpha} \ar[rd]|(0.5){\alpha} &
{} & \comp[I]{1}{K}
\\ {} & \comp[I]{1}{K} \ar[ru]|(0.5){\gamma} & {}  }
\end{equation}

The geometric decomposition of $\Gamma$ is the union of the two
full quivers\begin{equation}\label{squarenot0}
\xymatrix@C=30pt@R=2pt{
{} & {} & \comp[I]{1}{K} \ar[rd]|(0.5){0} & {} & {}
\\ \Gamma_1 = & \comp[I]{1}{K} \ar[ru]|(0.5){\alpha} \ar[rd]|(0.5){\alpha} &
{} & \comp[I]{1}{K} & ,{}
\\ {} & {} & \comp[I]{1}{K} \ar[ru]|(0.5){\gamma} & {} & {} }
\xymatrix@C=30pt@R=2pt{
{} & {} & \comp[I]{1}{K} \ar[rd]|(0.5){\beta} & {}
\\ \Gamma_2 = & \comp[I]{1}{K} \ar[ru]|(0.5){\alpha} \ar[rd]|(0.5){\alpha} &
{} & \comp[I]{1}{K} & {}
\\ {} & {} & \comp[I]{1}{K} \ar[ru]|(0.5){0} & {} & {} }
\end{equation}
On the other hand, the subalgebra of $A$ corresponding to the full
quiver
\begin{equation}\label{EE21}
\xymatrix@C=30pt@R=2pt{
{} & \comp[I]{1}{K} \ar[rd]|(0.5){\beta} & {}
\\ \comp[I]{1}{K} \ar[ru]|(0.5){\alpha} \ar[rd]|(0.5){\alpha} &
{} & \comp[I]{1}{K}
\\ {} & \comp[I]{1}{K} \ar[ru]|(0.5){-\beta} & {}  }
\end{equation}
has the property that the radical squared is 0, which fails in
each of the algebras $A_1,A_2$ of \eq{squarenot0}, and it follows
that $A$ cannot be a subdirect product of $A_1$ and $A_2$ as
algebras.
\end{exmpl}

Nevertheless, the algebra $A$ of Example~\ref{mess} is a subdirect
product of $A_1$ and $A_2$ as vector spaces, and the following
observation is immediate:

\begin{rem}\label{sep} If the full quiver $\Gamma$ of an algebra $A$ has a pair of
separated arrows, then $A$ has subalgebras $A_i$ corresponding to
the $\Gamma_i$ in the geometric decomposition of $\Gamma,$ such
that $A$ is a subdirect product of the $A_i$ as vector spaces.
\end{rem}

Also, note that the obstruction in Example~\ref{mess} to obtaining
$A$ as a subdirect product of algebras came from a subalgebra of
$A$ whose full quiver has extra gluing, which would not occur from
 elements in ``generic'' position. Thus, we will pursue these
 decomposition when studying polynomial identities in \cite{BRV3},
 \cite{BRV4}.

Remark ~\ref{sep1} can be refined further using
Lemma~\ref{permut2}.
 Let us embellish  this
discussion with a few more examples.

\begin{exmpl}\label{Grass2}
\begin{enumerate}
\item Consider the Grassmann algebra with two generators, whose
full quiver is given in \eqref{G2+}. Since we are interested in
its radical $J$, which is the algebra without 1, we use the full
quiver~\eqref{G2.1}. Let $J_1$ be the subspace corresponding to
the arrows $\alpha$, and $J_2$ be the subspace corresponding to
the arrows $\beta$. There are also two arrows (denoted $\gamma$)
corresponding to paths of length 2, which correspond to a space
$J_3$: \vskip 0.2cm
\begin{equation}\label{G21}
\xymatrix@C=40pt@R=32pt{
\circ \ar@/^0pt/[r]^{\alpha} \ar@/_0pt/[dr]_{-\gamma}
&
\circ \ar@/^0pt/[d]^{-\beta}
\\
&
\circ
} \qquad
\xymatrix@C=40pt@R=32pt{
\circ \ar@/_0pt/[d]^{\beta}\ar@/_0pt/[dr]^{\gamma}
&
\\
\circ \ar@/^0pt/[r]^{\alpha}
&
\circ
}
\end{equation}
Thus, $J_1$, $J_2$, and $J_3$ correspond respectively to the
operator spaces
\begin{equation}\label{G22}
\xymatrix@C=40pt@R=15pt{
\circ \ar@/^0pt/[r]^{\alpha}
&
\circ
\\
\circ \ar@/^0pt/[r]^{\alpha}
&
\circ
}, \qquad
\xymatrix@C=20pt@R=32pt{
\circ \ar@/_0pt/[d]^{\beta}
&
\circ \ar@/^0pt/[d]^{-\beta}
\\
\circ
&
\circ
},\qquad
\xymatrix@C=40pt@R=32pt{
\circ \ar@/^0pt/[dr]^{\gamma}
&
\\
&
\circ
}.
\end{equation}

\item The path \begin{equation}\label{33}
\xymatrix@C=30pt{
\circ \ar@/^0pt/[r]^{\alpha}
&
\circ \ar@/^0pt/[r]^{\beta}
&
\circ \ar@/^0pt/[r]^{\alpha +\beta}
&
\circ {}
}
\end{equation}
 decomposes as \begin{equation}
\xymatrix@C=30pt{
\circ \ar@/^0pt/[r]^{\alpha}
&
\circ
&
\circ \ar@/^0pt/[r]^{\alpha}
&
\circ {}
}, \qquad\qquad \xymatrix@C=30pt{
\circ
&
\circ \ar@/^0pt/[r]^{\beta}
&
\circ \ar@/^0pt/[r]^{\beta}
&
\circ {}
},
\end{equation}
and the composition of two arrows (corresponding to $J^2$) is
\begin{equation}
\xymatrix@C=30pt{
\circ \ar@/^10pt/[rr]^{\alpha\beta}
&
\circ \ar@/^10pt/[rr]|(0.25)\hole^{\alpha\beta}
&
\circ
&
\circ {}
}, \qquad\qquad \xymatrix@C=30pt{
\circ
&
\circ \ar@/^10pt/[rr]^{ \beta^2}
&
\circ
&
\circ {}},
\end{equation}
noting that $\alpha^2 = 0$. The composition of three arrows
(corresponding to $J^3$) is
\begin{equation}
\xymatrix@C=30pt{
\circ \ar@/^15pt/[rrr]^{\alpha\beta(\alpha+\beta)}
&
\circ
&
\circ
&
\circ {}
}.\end{equation}
\end{enumerate}
\end{exmpl}

\begin{rem} This leads us to the question of how to ``restore'' the
imprimitive arrows in the full quiver. The answer is obtained by
considering the full quiver of powers of the radical. More
generally, if $A$ is an algebra without 1, we call an arrow {\bf
pared} if both of its vertices are $\circ$. Then any path $\circ
\to \circ \to \cdots \to \circ$ of length $\ell$ comprised
entirely of pared arrows corresponds to a primitive (pared) arrow
in $A^{\ell}$.
(This only works for pared arrows, since an element of $J$
matching an arrow having a vertex corresponding to an idempotent
of $A$ can be multiplied as many times as one wants by that
idempotent and remain the same unpared arrow of $A^\ell$ for each
$\ell$. Thus, the full quiver of $A^\ell$ is comprised of all
unpared arrows of the full quiver of $A$, as well as paths of
length $\ell$ of pared arrows.)

For example, if $A$ is the algebra whose full quiver is
$$
\xymatrix@C=30pt{
\circ \ar[r]
&
\circ \ar[r]
&
\circ \ar[r]
&
\bullet \ar[r] & \circ }$$ then the longest path of pared arrows
has length $3$. Here $A^2$ corresponds to the sub-quiver
$$
\xymatrix@C=30pt{
\circ \ar@/^10pt/[rr]
&
\circ \ar@/^10pt/|(0.25){\hole}[rr]
&
\circ \ar[r]
&
\bullet \ar[r] & \circ },$$ and $A^n$ (for every $n \geq 3$)
correspond to the sub-quiver \vskip0.1cm
$$
\xymatrix@C=30pt{
\circ \ar@/^15pt/[rrr]
&
\circ \ar@/^10pt/[rr]
&
\circ \ar[r]
&
\bullet \ar[r] & \circ },$$ while $J$, $J^2$ and $J^3$ correspond
to
$$
\xymatrix@C=30pt{
\circ \ar[r]
&
\circ \ar[r]
&
\circ \ar[r]
&
\circ \ar[r] & \circ },$$
$$
\xymatrix@C=30pt{
\circ \ar@/^10pt/[rr] \ar@/^15pt/[rrr]
&
\circ \ar@/^10pt/|(0.25){\hole}[rr] \ar@/^15pt/|(0.162){\hole}|(0.333){\hole}[rrr]
&
\circ \ar@/^10pt/|(0.25){\hole}|(0.3){\hole}[rr]
&
\circ & \circ },$$ and
$$
\xymatrix@C=30pt{
\circ \ar@/^15pt/[rrr] \ar@/^20pt/[rrrr]
&
\circ \ar@/^15pt/|(0.333){\hole}[rrr]
&
\circ
&
\circ & \circ },$$ respectively.

\end{rem}

\subsection{Idempotents and sub-quivers}\label{ssec:sd}

\begin{rem}
Let $\Gamma_0 \sub \Gamma$, and $e = e(\Gamma_0)$. Then $eAe$ is a
subalgebras of $\M[n](K)$, whose full quiver is $\Gamma_0$. The
relations of $eAe$ can be obtained from those of $A$ by removing
all coefficients corresponding to vertices not in $\Gamma_0$. (We
emphasize again that in general $eAe$ is not a subalgebra of $A$).
\end{rem}

The obstacles at the algebra level for ``improving'' quivers can
be overcome only at the generic level, as we see in the next two
examples.

\begin{exmpl}\label{sdirr}
Consider the \Zcd\ algebra $A$ whose full quiver is
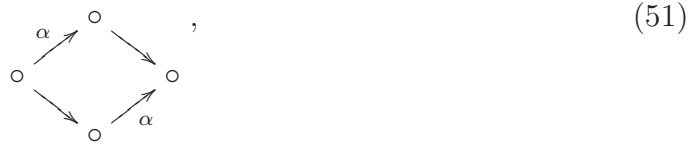
\begin{equation}
\xymatrix@C=16pt@R=10pt{
{} & \circ \ar[rd] & {} \\
\circ \ar[ur]^{\alpha} \ar[dr] & & \circ \\
{} & \circ \ar[ru]_{\alpha} & {} },
\end{equation}
which is geometrically decomposed into two paths. Taking $x =
e_{12}+e_{34}$, $y = e_{13}$ and $z = e_{24}$ we obtain  the
presentation $A = K[x,y,z \,|\, x^2=xy=y^2=yz=zx=zy=z^2=0, yx =
xz]$ (as an algebra without unit), and one easily checks that in
any non-trivial quotient, $yx = 0$ (which, incidentally, can be
described in the same quiver, by imposing the relation that the
imprimitive diagonal arrow is zero). Therefore, this algebra $A$
is subdirectly irreducible.
\end{exmpl}

\begin{exmpl}
Continuing with the algebra of \Eref{sdirr}, the algebra of
generic elements has the presentation $\hat{A} = K[x,y,z \,|\,
x^2=xy=y^2=yz=zx=zy=z^2=0]$; namely, the relation $yx = xz$ no
longer holds, and the quiver separates to
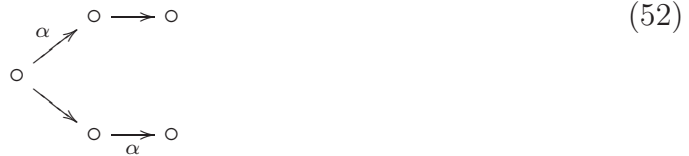
\begin{equation}
\xymatrix@C=16pt@R=10pt{
{} & \circ \ar[r] & \circ \\
\circ \ar[ur]^{\alpha} \ar[dr] & & { } \\
{} & \circ \ar[r]_{\alpha} & \circ
}
\end{equation}
which  has a geometric
 decomposition into two paths.
\end{exmpl}

\subsection{Appendix: Pseudo-quivers}

Change of basis within a glued component can simplify the full quiver of the representation. We give few examples, and explain how they lead us to the notion of a pseudo-quiver.

\begin{exmpl}\label{pseudo1}
Consider the full quivers
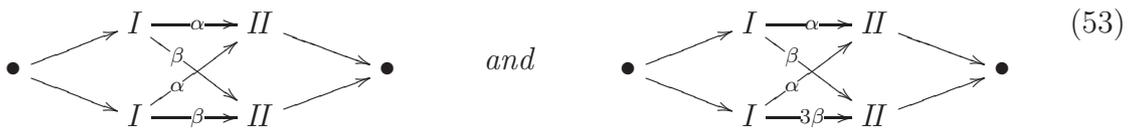
\begin{equation}\label{G9}
\xymatrix@C=32pt@R=2pt{
{} & \comp[I]{1}{K} \ar[r]|(0.5){\alpha}
\ar[rdd]|(0.34){\beta}
& \comp[\II]{1}{K}
\ar[rd]|(0.5){} & {}
\\ \comp{1}{K} \ar[ru] \ar[rd] &
{} & {} & \comp{1}{K}
\\ {} & \comp[I]{1}{K} \ar[ruu]|(0.34){\alpha} \ar[r]|(0.5){\beta} & \comp[\II]{1}{K} \ar[ru] & {} }
\qquad
\begin{array}{c} {\ } \\ {\ } \\ \mbox{and} \end{array}
\qquad
\xymatrix@C=32pt@R=2pt{
{} & \comp[I]{1}{K} \ar[r]|(0.5){\alpha}
\ar[rdd]|(0.34){\beta}
& \comp[\II]{1}{K}
\ar[rd]|(0.5){} & {}
\\ \comp{1}{K} \ar[ru] \ar[rd] &
{} & {} & \comp{1}{K}
\\ {} & \comp[I]{1}{K} \ar[ruu]|(0.34){\alpha} \ar[r]|(0.5){3\beta} & \comp[\II]{1}{K} \ar[ru] & {} }
\end{equation}
In the left-hand side of \eq{G9}, a base change of the components denoted by $I$, which replaces $e_{22},e_{33}$ with $e_{22},e_{33}-e_{22}$ along the lines of \Eref{pseud}, results in the full quiver to the left of \eq{G9+}, which is subdirectly reducible, with no gluing in each component. However in the right-hand side of \eq{G9}, the same base change results in the quiver to the right of \eq{G9+}, which still has gluing.
\begin{equation}\label{G9+}
\xymatrix@C=32pt@R=2pt{
{} & \comp[I]{1}{K} \ar[r]|(0.5){\alpha} \ar[rdd]|(0.34){\beta}
& \comp[\II]{1}{K}
\ar[rd]|(0.5){} & {}
\\ \comp{1}{K} \ar[ru] \ar[rd] &
{} & {} & \comp{1}{K}
\\ {} & \comp[I]{1}{K}
& \comp[\II]{1}{K} \ar[ru] & {} }
\qquad
\begin{array}{c} {\ } \\ {\ } \\ \mbox{\ ; \ } \end{array}
\qquad
\xymatrix@C=32pt@R=2pt{
{} & \comp[I]{1}{K} \ar[r]|(0.5){\alpha} \ar[rdd]|(0.34){\beta}
& \comp[\II]{1}{K}
\ar[rd]|(0.5){} & {}
\\ \comp{1}{K} \ar[ru] \ar[rd] &
{} & {} & \comp{1}{K}
\\ {} & \comp[I]{1}{K}
\ar[r]|(0.5){2\beta} & \comp[\II]{1}{K} \ar[ru] & {} }
\end{equation}
\end{exmpl}

\begin{exmpl}
Now consider the full quiver
\begin{equation*}
\xymatrix@C=40pt@R=0pt{
{} &  \comp[I]{1}{K} \ar@{->}|(0.5){\alpha}[rd] \ar@{->}|(0.25){\alpha}[rddd] & {} & {} \\
{} &  {} & \comp[\II]{1}{K} \ar@{->}[rd] & {} \\
\comp{1}{K} \ar@{->}[r] \ar@{->}[ruu] \ar@{->}[rdd] &  \comp[I]{1}{K} \ar@{->}|(0.3){\beta}[ru] \ar@{->}|(0.3){2\beta}[rd] & {} & \comp{1}{K}, \\
{} &  {} & \comp[\II]{1}{K} \ar@{->}[ru] & {} \\
{} &  \comp[I]{1}{K} \ar@{->}|(0.25){\gamma}[ruuu] \ar@{->}|(0.5){3\gamma}[ru] & {} & {}
}
\end{equation*}
in which three pairs of arrows may be traded along the lines of \Eref{pseud}. If not for the gluing, we could perform a base change resulting in the full quiver
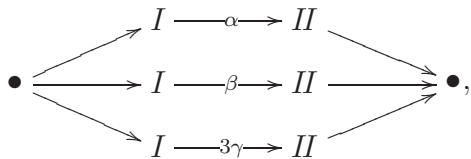
\begin{equation*}
\xymatrix@C=40pt@R=0pt{
{} &  \comp[I]{1}{K} \ar@{->}|(0.5){\alpha}[r]  & \comp[\II]{1}{K} \ar@{->}[rdd] & {} \\
{} &  {} & {} & {} \\
\comp{1}{K} \ar@{->}[r] \ar@{->}[ruu] \ar@{->}[rdd] &  \comp[I]{1}{K} \ar@{->}|(0.5){\beta}[r]  & \comp[\II]{1}{K} \ar@{->}[r] & \comp{1}{K}, \\
{} &  {} & {} & {} \\
{} &  \comp[I]{1}{K} \ar@{->}|(0.5){3\gamma}[r] & \comp[\II]{1}{K} \ar@{->}[ruu] {} & {}
}
\end{equation*}
However the identical gluing of the $I$ components carries over to a linear dependence among the vertices denoted by $\II$, because otherwise the algebra would degenerate to the zero algebra.
\end{exmpl}

Recall that in a representations in Wedderburn block form, the diagonal blocks are direct sums of matrix algebras (over extensions of $F$). Gluing replaces the natural central idempotents of the glued blocks by their sum, and specifies a `glued' subspace of the direct sum, such as $F(1,1,1) \sub F\oplus F \oplus F$.

Without diagonal gluing, a base change may simplify the arrows in the full quiver in a way that reduces the number of arrows. However such a base change may not preserve the diagonal gluing, resulting in a representation in which the diagonal blocks corresponding to vertices may have extra linear relations. Such a representation, which is no longer in Wedderburn block form, can be encoded in a full quiver with quais-linear gluing among the vertices, which we call a \textbf{pseudo-quiver}.

Let us see how we can improve the form of the pseudo-quiver by means of
starting with the `correct' representation.

\begin{prop}\label{basicquiv2}
Suppose all off-diagonal gluing of the  pseudo-quiver of a \Zcd\
algebra is proportional. Then the algebra has a representation whose
pseudo-quiver (with respect to a suitable basis) has no linear
dependence relations among arrows all starting or ending   at the
same vertex.
\end{prop}
\begin{proof} For any $a \in A$, we define its \textbf{level} to
be the smallest $k$ for which $aJ^k = 0.$ Suppose $m$ is the
nilpotence index of $J$; i.e., $J^m = 0$ but $J^{m-1} \ne 0.$ We
take a base of the functionals corresponding to arrows, starting
with the top level $m-1$, and continuing inductively (which is
possible, since the graph of the pseudo-quiver has no cycles).

View the arrows from the   vertex $v$ as linear functionals. By
definition, different arrows originating from the same vertex
correspond to independent linear functionals. In this way, we see
that any purely proportionally glued arrows originating from the
vertex $v$ correspond to the same functional, and thus can be
combined to a single arrow after a linear change of basis. The
dual argument then applies to two arrows ending both at the same
vertex.
In this manner,
we can eliminate all linear dependence relations among the arrows.
\end{proof}

\begin{defn} \textbf{Making an open sandwich} is the process of starting with the vertex $v$ and applying
arrows corresponding to base elements of $J$, viewed as linear
functionals as in the proof.\end{defn}

\begin{cor}\label{Canon1}
When the base field $F$ is infinite, every algebra $A$ has a
representation in which the linear functionals corresponding to
the arrows of their pseudo-quivers are linearly independent.
\end{cor}
\begin{proof} We may start with a representation in whose pseudo-quiver
the relations among the arrows are given by linear dependences, in view of
\cite[Proposition~4.7.(2)]{BRV1} (noting that quasi-linear relations
are linear since $F$ is infinite). By Theorem~\ref{propglu}, we
have a representation whose pseudo-quiver can only have purely
proportional gluing, which we eliminate using
Proposition~\ref{basicquiv2}.
\end{proof}

In view of Proposition~\ref{basicquiv2}, when studying relatively
free algebras, we restrict our attention to pseudo-quivers that
have only proportional Frobenius gluing among the arrows, and the
only possible gluing among vertices comes from quasi-linear
relations. In view of Corollary~\ref{Canon1}, when $F$ is
infinite, we can always find a representation whose pseudo-quiver
never has gluing from different arrows
 emerging from the same vertex  (or, dually, leading to the
same vertex).

Although \Cref{Canon1} may seem a technical curiosity, sacrificing
the independence of the blocks for the (perhaps less desirable)
linear independence of the arrows, it turns out to be a key tool
in  combinatorial PI-theory. Here is an application to the general
structure theory.

\begin{thm}\label{nilpcan} Any \Zcd\ algebra $A$ has a representation for whose full quiver the maximal length of
the branches equals the index of nilpotence of the radical minus
~1.\end{thm}
\begin{proof}
Let $J$ be the radical, and take  $m$ such that $J^m = 0$ but
$J^{m-1} \ne 0.$  Recalling from the proof of
Proposition~\ref{basicquiv2} that the arrows of the pseudo-quiver
correspond to a base of linear functionals, we can make an open sandwich
by choosing one branch and zeroing out the other branches; in this
way we eliminate cancellation among different branches.

In order to pass from pseudo-quivers to quivers, we consider the
representation giving rise to the pseudo-quiver, and take its
corresponding full quiver. Note that the length of the maximal
path of the pseudo-quiver is the same as the length of the maximal
path of our full quiver.
\end{proof}

This proof works for pseudo-quivers, but not necessarily for
quivers, since we need the linear independence of the arrows.

\begin{exmpl} The full quiver of the Grassmann algebra on two generators
(cf.~\Eref{Grass2}) is also its pseudo-quiver. The radical
$J$ is $Fe_1 + Fe_2$. One open sandwich which we get is $Fe_1e_2\ne 0.$

On the other hand in the algebra of \eq{EE21}, the only open sandwich is composed of $e_1$ and $e_2+e_3$, as $\beta$ maps $e_2+e_3$ to zero, and indeed in this algebra $J^2 = 0$.
\end{exmpl}

As we shall see in \cite{BRV3}, when testing whether
a polynomial $f(x_1, \dots, x_d)$ is a polynomial identity of an
algebra $A$, one observes that the monomials of $f$ can be
determined to be zero or nonzero on~$A$ according to their
evaluations on paths of a suitable pseudo-quiver, whereas the
analogous assertion for full quivers fails. Also, pseudo-quivers
naturally give rise to PIs, as is to be seen in \cite{BRV3}.

\end{document}